\documentclass[12pt, oneside]{amsart}
\usepackage[utf8]{inputenc}

\usepackage{multirow}
\usepackage{amssymb,amsmath,amsfonts,amsthm}
\usepackage{hyperref}
\usepackage{verbatim}

\newcommand{\prk}{\operatorname{prk}}

\newcommand{\Aut}{\operatorname{Aut}}

\newcommand{\Inn}{\operatorname{Inn}}
\newcommand{\Inndiag}{\operatorname{Inndiag}}

\numberwithin{equation}{section}

\newtheorem{lemma}{Lemma}[section]
\newtheorem{theorem}{Theorem}

\theoremstyle{remark}
\newtheorem{rem}{Remark}[section]

\begin{document}


\title[Recognition of finite simple groups by element orders]{\textsc{The problem of recognition of finite simple groups by element orders is solved}}

\author{\textsc{M.A. Grechkoseeva}}
\address{Novosibirsk State University, Pirogova, 1, Novosibirsk 630090, Russia}
\email{grechkoseeva@gmail.com}

\author{\textsc{A.M. Staroletov}}
\address{
Sobolev Institute of Mathematics, Koptyuga 4, Novosibirsk 630090, Russia}
\email{staroletov@math.nsc.ru}

\author{\textsc{A.V. Vasil'ev}}
\address{Novosibirsk State University, Pirogova, 1, Novosibirsk 630090, Russia;\newline \hspace*{3mm}  Sobolev Institute of Mathematics, Koptyuga 4, Novosibirsk 630090, Russia}
\email{vasand@math.nsc.ru}

\begin{abstract}

For a finite group $G$, let $\omega(G)$ be the set of element orders of $G$ and let $h(G)$ be the number of pairwise nonisomorphic finite groups $H$ with $\omega(H)=\omega(G)$. We say that the recognition problem is solved for $G$ if the number $h(G)$ is known, and if it is finite, then all finite groups $H$ with $\omega(H)=\omega(G)$ are listed. We complete the solution of the recognition problem for all finite simple groups.
\smallskip

\end{abstract}

\begingroup
\def\uppercasenonmath#1{} 
\let\MakeUppercase\relax 
\maketitle
\endgroup

\section{Introduction}

Given a group $G$, we denote by $\omega(G)$ the set of element orders of~$G$, that is,
$$
\omega(G)=\{|x| : x\in G\}.
$$
Following \cite{84Ady.t}, we refer to this set as the {\em spectrum} of~$G$. Groups with the same spectrum are said to be {\em isospectral}.
General question that we are interested in is what can be said about groups isospectral to a given group~$G$.

In this article, we restricted ourselves to the case when $G$ and groups isospectral to $G$ are finite. The deep and difficult question of recovering of the properties of an arbitrary group from its spectrum lies outside the scope of this paper, and we refer the reader to \cite{18HLM, 14LytMaz, 09MazShi}.

For a finite group $G$, we denote by~$h(G)$ the number of pairwise nonisomorphic finite groups isospectral to~$G$.  We say that the {\em recognition problem is solved} for $G$ if the number $h(G)$ is known, and if it is finite, then all the finite groups isospectral to $G$ are described. A~group $G$ is called {\em recognizable} ({\em by spectrum}) if $h(G)=1$ and {\em almost recognizable} if $h(G)<\infty$.

Observe that $h(G)$ is finite only if $G$ does not have nontrivial normal abelian subgroups. More precisely, $h(G)=\infty$ if and only if $G$ is isospectral to a group with nontrivial normal abelian subgroup~\cite{12MazShi.t}. Thus, the recognition problem is solved for groups with nontrivial solvable radical.

Most basic groups with trivial solvable radical are definitely nonabelian simple groups. The purpose of this article is to provide the solution of the recognition problem for all finite simple groups. Before formulating the main result, we briefly touch upon the history of the issue (see also the recent survey \cite{23Survey} and the references therein).

It is not easy to pinpoint the exact moment when the first results allowing to recognize finite simple groups appeared. At the turn of the 19th and 20th centuries, Burnside \cite{99Bur} described the finite groups in which all elements of even order are involutions (see also \cite[Theorem XI.2.7]{82HupBl3}, where this result was attributed to Brauer, Suzuki and Wall's paper of 1958). One can easily deduce from this result that $h(PSL_2(q))=1$ for all even $q>2$.

Another result of this kind can be extracted from the classification of finite groups whose spectra consist of prime powers only (the EPPO-groups, for brevity). The study of these groups was started in 1957 by Higman \cite{57Hig}, simple groups among them were found by Suzuki \cite[Theorem~16]{62Suz} in his seminal paper of 1962, the classification of EPPO-groups was completed by Brandl  \cite{81Bra} in 1981 (with one minor inaccuracy, the group $M_{10}$ was omitted). The classification implies that seven of eight nonabelian simple EPPO-groups are uniquely determined by their spectra in the class of all finite groups. The only exception is the group $A_6\simeq PSL_2(9)$ with spectrum $\{1,2,3,4,5\}$. As readily seen, the same spectrum has the semidirect product $V\rtimes SL_2(4)$, where $V$ is the natural $2$-dimensional ${SL}_2(4)$-module. Thus, $h(A_6)=\infty$.

The Chinese mathematician Shi was apparently the first one who explicitly formulated the problem of recognizing simple group by its element orders. In the middle of 1980s, studying the EPPO-groups, Shi \cite{84Shi, 86Shi} noticed that groups $A_5$ and $PSL_2(7)$ are uniquely characterized by their spectra in the class of finite groups, and he asked which finite simple groups have the same property. That was the time when the classification theorem (CFSG) was announced and the complete list of simple groups became available. Let us recall here, for convenience, that according to the classification, see, e.g., \cite{94GorLySol}, the finite simple groups are exactly the following groups:
\begin{enumerate}
\item the groups of prime order;
\item the alternating groups $A_n$, $n\geq5$;
\item the simple classical groups $L_n(q)$, $n\geq2$; $U_n(q)$, $n\geq3$; $S_{2n}(q)$, $n\geq2$; $O_{2n+1}(q)$ $n\geq3$; $O^{\pm}_{2n}(q)$, $n\geq4$;
\item the simple exceptional groups of Lie type $G_2(q)$, $F_4(q)$, $E_6(q)$, $E_7(q)$, $E_8(q)$, ${}^2B_2(q)$, ${}^3D_4(q)$, ${}^2G_2(q)$, ${}^2F_4(q)$, ${}^2F_4(2)'$, ${}^2E_6(q)$;
\item the 26 sporadic groups.
\end{enumerate}

Note that in this list and further in the text we use the one-letter notation for simple classical groups from the \emph{Atlas of Finite Groups} \cite{85Atlas}.

Returning to the recognition problem, since there are nonabelian simple groups $L$ with $h(L)=\infty$, two ways to handle the problem appear. The first one is to add  another condition to the equality $\omega(G)=\omega(L)$. In 1987, Shi conjectured \cite{89ShiS} that every nonabelian simple group is uniquely determined by its spectrum and order in the class of all finite groups. Shi's conjecture turned to be right. This was finally proven in 2009, see \cite{09VasGrMaz1.t}.


Another way is to look at all groups $G$ isospectral to a given simple group $L$. In the middle of 2000s, Mazurov proposed a conjecture
whose essence can be formulated as follows: If $L$ is a nonabelian simple group, then for a reasonable number of exceptions, each of which can be explicitly described, the equality $\omega(G)=\omega(L)$ implies that, up to isomorphism,
\begin{equation}\label{eq:ar}
L\leq G\leq \Aut L,
\end{equation}
where, as usual, $L$ is identified with $\Inn L$. Recall that a group $G$ satisfying \eqref{eq:ar} is called an almost simple group with socle $L$.

The main result of the present article shows that Mazurov was right: in most cases, even without the condition $|G|=|L|$, a group $G$ isospectral to $L$  must be very close to~$L$.

%

\begin{theorem}\label{t:solution}
Suppose that $L$ is a finite nonabelian simple group. \begin{enumerate}
\item If $L$ is one of the groups $A_6$, $A_{10}$, $J_2$, $^3D_4(2)$, $L_3(3)$, $U_3(5)$, $U_3(q)$, where $q$ is a~Mersenne prime such that $q^2-q+1$ is also a prime, $U_5(2)$, $S_4(3)$, $S_4(q)$, $q\neq 3^{2k+1}$ with $k\geq 1$, $S_8(q)$,  $q\neq 7^{k}$, or $O_9(q)$, then $h(L)=\infty$.
\item If $L$ is one of the groups $S_6(2)$, $O_8^+(2)$, $O_7(3)$, $O_8^+(3)$, then $h(L)=2$, with
$\omega(S_6(2))=\omega(O_8^+(2))$ and $\omega(O_7(3))=\omega(O_8^+(3))$.

\item Otherwise, every finite group $G$ with $\omega(G)=\omega(L)$ is an almost simple group with socle isomorphic to~$L$ and all such $G$ are known.
\end{enumerate}
\end{theorem}

\begin{rem}\label{r:Mers}
At present, 52 Mersenne primes are known \cite{GIMPS}. The only primes $q$ among them such that $q^2-q+1$ is also a prime are $3$ and~$7$, and the conjecture is that there are no other such primes, see also \cite[Remark~1.1]{26GSV_arxiv} and \cite{06Zav.t}.
\end{rem}

\begin{rem}\label{r:i} For each group $L$ from item (i) of the theorem, \cite[Table~10]{23Survey} provides an example of a group $G$ such that $G$ is isospectral to $L$ and includes a nontrivial normal abelian subgroup. As we know from the above discussion, such an example guarantees that $h(L)=\infty$. The groups $L_3(3)$, $U_3(3)$, and $S_4(3)$ are the only nonabelian simple groups isospectral to solvable groups, see, e.\,g., \cite[Theorem~3.3]{23Survey} and the references before it. For every other group $L$ from item (i) of the theorem, each group $G$ isospectral to $L$ has the normal series $1\leq K\leq H\leq G$, where $K$ is the solvable radical of $G$, $H/K\simeq S$ is a nonabelian simple group, and $G/K$ is an almost simple group with socle $H/K$. 
Furthermore, $G/H$ is cyclic \cite{22GreVas} and, if $L\neq A_{10}$, then $K$ is nilpotent \cite{20YanGrVas}.
\end{rem}

\begin{rem}\label{r:ii} For a nonabelian simple group $L$, there is an alternative way to define almost recognizability. Namely, $L$ is called {\em almost recognizable} if every finite group $G$ isospectral to $L$ satisfies \eqref{eq:ar}, cf., e.\,g., \cite{26GSV_arxiv}. As Theorem~\ref{t:solution} shows, the only nonabelian simple groups with $h(L)<\infty$ that are not almost recognizable in that sense are four groups from item (ii) of the theorem.
\end{rem}

\begin{rem} If $L$ is an almost recognizable simple group from item (iii) of the theorem, then the precise value of $h(L)$ and the description of all groups $G$ isospectral to $L$ can be readily extracted from \cite[Tables~1--9]{23Survey}. In the case when $L$ is classical group (Tables~1--6), one should proceed as follows: in the table corresponding to the type of $L$, find the section containing the condition of the form ``$n\geq n_0$'' and take $h(L)$ and the description of isospectral groups from this section. We decided not to repeat Tables 1--9 of \cite{23Survey} in the present article since they need to be supplemented by explanations of how to use them from \cite[Section 2.3]{23Survey} and 
notation for outer automorphisms from \cite[Section 2.2]{23Survey}. Nevertheless, let us observe that $G/L$ is cyclic for each group $G$ from item (iii) of the theorem. Moreover, $G/L=1$ provided $L$ is alternating or sporadic. This is not generally true when $L$ is a group of Lie type, say, $h(L_3(7^{3^k}))=k+1$ for each $k\geq0$, so $h(L)$ can take as a value any positive integer.
\end{rem}

\begin{rem} By \cite[Corollary~4.1]{23Survey}, if $G$ is a finite group with $h(G)<\infty$, then every locally finite group isospectral to $G$ must be finite. It follows that Theorem~\ref{t:solution} also provides the characterization of finite simple groups by their spectra in the class of locally finite groups.
\end{rem}

From the above discussion, it is clear that Theorem~\ref{t:solution} is a final point of a long journey (the 2023 survey \cite{23Survey} already cites more than 100 papers by about 30 mathematicians on this very subject). The real impact of the present article is presented in Theorem~\ref{t:main} below. To state it, we need some additional terminology. For a finite group $G$, let $\pi(G)$ be the set of prime divisors of the order of $G$ and denote by $t(G)$ the greatest size of a subset $\rho$ of $\pi(G)$ such that $pq\not\in\omega(G)$ for any two distinct primes $p,q\in\rho$. Note that the dimension of a simple classical group $L$ and $t(L)$ are linearly equivalent, see Table~\ref{tab:tL} in Section~\ref{s:pre}; furthermore, the exact dimensions of symplectic and orthogonal groups with $6\leq t(L)\leq 13$ can be found in Table~\ref{tab:6tL14}.

\begin{theorem}\label{t:main}
Suppose that $L$ is a finite nonabelian simple symplectic or orthogonal group over field of odd characteristic and $6\leq t(L)\leq 13$. Then every finite group isospectral to $L$ is an almost simple group with socle isomorphic to~$L$.
\end{theorem}

Here we explain why to complete the solution of the recognition problem for all finite simple groups, it suffices to prove such a particular assertion (an even much more particular one, see the first paragraph of Section~\ref{s:structure}).

The solution of the recognition problem for sporadic groups was completed by Mazurov and Shi in 1998 \cite{98MazShi}, and for the alternating groups $A_n$, $n\geq5$, by Gorshkov in 2013 \cite{13Gor.t}. We refer the reader to \cite{98MazShi} and \cite{13Gor.t} for history of the subject and relevant links.

Now we assume that $L$ is a simple group of Lie type. Observe first that the recognition problem for the groups from item (ii) of Theorem~\ref{t:solution} was solved in 1997 by Mazurov \cite{98Maz.t}, and Shi and Tang~\cite{97ShiTan}.

Let us turn now to the groups with $t(L)=2$ (note that $t(L)>1$ for all nonabelian simple groups, see \cite{05VasVd.t}). According to \cite[Table~4]{11VasVd.t}, there is only one simple exceptional group of Lie type satisfying this condition, the Steinberg triality group ${}^3D_4(2)$. In 2013, Mazurov shows that $h({}^3D_4(2))=\infty$ \cite{13Maz.t} and, running a little ahead, it turned out that this is the only exceptional group $L$ of Lie type with $h(L)=\infty$.

According to \cite[Table~2]{11VasVd.t}, the simple linear and unitary groups with $t(L)=2$ are found among the groups of dimension $3$ over the fields of odd order, in particular, the groups $L_3(3)$ and $U_3(3)$ satisfy this condition. The recognition of $L_3(q)$ and $U_3(q)$ was completed by 2006 in Zavarnitsine's papers \cite{04Zav1.t} and \cite{06Zav.t}.

In view of \cite[Table~3]{11VasVd.t}, except of the pair  isospectral groups $S_6(2)$ and $O_8^+(2)$ from item (ii) of Theorem~\ref{t:solution}, the groups $S_4(q)$ are the only
groups among the symplectic and orthogonal ones such that $t(L)=2$. The recognition of groups $S_4(q)$ was completed by Mazurov in~2002, see \cite{02Maz.t}.

Now $t(L)\geq3$. Then a group $G$ with $\omega(G)=\omega(L)$ must be nonsolvable. Otherwise, let $\rho$ be a subset of maximal size of primes in $\omega(L)=\omega(G)$ such that $pq\not\in\omega(L)$ for every distinct $p,q\in\rho$. If $G$ is solvable, it includes a Hall $\rho$-subgroup whose spectrum  consists of prime powers only. Since $|\rho|=t(L)\geq3$, this contradicts well-known Higman's theorem \cite[Theorem~1]{57Hig}.

On the other hand, the spectrum of every finite simple group $L$ of Lie type includes an odd prime $r$ such that $2r\not\in\omega(L)$ \cite[Theorem~7.1]{05VasVd.t}. In view of the main result of \cite{05Vas.t}, it follows that a group $G$ isospectral to $L$ has exactly one nonabelian composition factor $S$, so $G$ has the following structure:
\begin{equation}\label{eq:str}
1\leq K\leq H\leq G,
\end{equation}
where $K$ is a solvable radical of $G$, $H/K\simeq S$ is a nonabelian simple group, and $G/K$ is an almost simple group with socle $H/K$.

Assume first that the isomorphism $S\simeq L$ takes place. Then not only $G$ but also both $H$ and $G/K$ in \eqref{eq:str} are isospectral to $L$, with $H$ being a cover of $L$, that is, a group homomorphically mapped onto $L$, and $G/K$ being an almost simple group with socle $L$.  We say that $L$ is recognizable by spectrum among covers (automorphic  extensions) if every cover of $L$ (almost simple group with socle $L$) isospectral to $L$ is isomorphic to $L$.


It turned out that the list of finite simple groups that are not recognizable among their covers is very short. Namely, it includes  only the groups ${}^3D_4(2)$, $U_5(2)$, and $U_3(q)$, where $q$ is a Mersenne prime such that $q^2-q+1$ is also a prime, see \cite[Theorem~3.6]{23Survey} and the references before this theorem.

On the other hand, simple groups of Lie type are mostly not recognizable among automorphic extensions, so the problem of describing almost simple groups of Lie type isospectral to the socle arises. The solution of this problem was completed by 2018, the last steps were done by Zvezdina for exceptional groups of Lie type \cite{16Zve.t} and Grechkoseeva for classical groups \cite{18Gr.t}. This work substantially uses the description of the spectra of finite simple groups of Lie type systematicaly given by Buturlakin, see \cite{08But.t, 10But.t, 18But.t} and the references therein.

Thus, it remains to verify that $S\simeq L$ for every group $G$ with $\omega(G)=\omega(L)$ (based on the terminology of \cite{02AlKon.t}, we call this step the quasirecognizability problem for brevity). For the exceptional groups of Lie type, the solution of the quasirecognizability problem was completed by Staroletov and Vasil'ev in 2014, see \cite{14VasSt.t} and references therein. Thus, we further assume that $L$ is a classical group (with $t(L)\geq3$).

The fact that $S$ cannot be an alternating or sporadic group, with one exception provided by $\omega(U_5(2))=\omega(3^5:M_{11})$, was proved  for symplectic and orthogonal groups in \cite{09VasGrMaz.t} and  for linear and unitary groups in \cite{11VasGrSt.t}. Furthermore, in the same two articles the case when $S$ is a group of Lie type over the field of the same characteristic as $L$ was considered. It turned out that in this case, $S\simeq L$ if $L$ is linear or unitary \cite[Theorem~3]{11VasGrSt.t}. If $L$ is symplectic or orthogonal, the situation is more complicated, and the same assertion is not generally true. In \cite{15VasGr1}, it was shown that either $S\simeq L$, or $L\in\{S_{8}(q), O_9(q)\}$ and $S=O_8^-(q)$ (recall that the cases when $L=S_4(q)$ or $L$ is from item~(ii) of Theorem~\ref{t:solution} have been already considered). In fact,  $\omega(O_9(q))=\omega(V\rtimes \Omega_8^-(q))$ for all $q$ \cite{14GrSt}, including the case when $q$ is even and $O_9(q)\simeq S_8(q)$, and  $\omega(S_8(q))=\omega(V\rtimes(\Omega_8^-(q):\langle\gamma\rangle))$  for all $q\not\in\{2^k,7^k\}$ \cite{18Gr.t}, where $V$ is a natural $8$-dimensional $\Omega_8^-(q)$-module and $\gamma$ is a graph automorphism of $\Omega_8^-(q)$. Finally, the groups $S_8(7^k)$ are recognizable by spectrum \cite{25Gr.t}.


The last and most difficult case is when $S$ is a group of Lie type but the characteristics of $S$ and $L$ differ. In 2015, Vasil'ev proved  that for every simple classical group $L$ with $t(L)\geq23$, a group $G$ isospectral to $L$ cannot have a composition factor isomorphic to a simple group of Lie type whose characteristic distinct from that of~$L$ \cite{15Vas}. This result made it possible to prove that all simple classical groups of sufficiently large dimension are almost recognizable, see \cite[Theorem~1.1]{15VasGr1}. Two years later Staroletov  \cite{17Sta} lowered the bound 23, showing that the same holds true if $t(L)\geq14$.

Thus, the recognition problem was reduced to the case when $t(L)\leq13$. The analysis of the remaining groups turned out to be quite technically complicated and took about ten years. The research was naturally divided in two directions. On the one hand, Staroletov \cite{21Sta.t,26Sta.t} solved  the recognition problem for simple linear and unitary groups with $t(L)\geq6$ and showed that $t(L)\leq t(S)\leq t(L)+2$ for symplectic and orthogonal groups with $5\leq t(L)\leq13$. On the other hand, Grechkoseeva and her colleagues completed the solution of recognition problem for the simple classical groups with $t(L)\leq4$, see \cite{19GrVasZv,20GrZv.t,25Gr.t,24GrPan.t,26Pan.t,26GreRod_arxiv}. The case $t(L)=5$ for all classical groups was quite recently resolved by Grechkoseeva, Staroletov and Vasil'ev  in~\cite{26GSV_arxiv}.  As readily seen from the above, this result completes, in particular, the solution of the recognition problem for the simple linear and unitary groups. Thus, Theorem~\ref{t:main} handling the case of the symplectic and orthogonal groups with $6\leq t(L)\leq13$ indeed finishes the proof of Theorem~\ref{t:solution}.

\section{Preliminaries}\label{s:pre}

\subsection{Number-theoretical results}

Let $a$ and $b$ be nonzero integers and $r$ a prime. The greatest common divisor and least common multiple of $a$ and $b$ are denoted by $(a,b)$ and $[a,b]$, respectively. The set of primes dividing $a$ is denoted by $\pi(a)$. We write $(a)_r$ for the  the highest power of $r$ dividing $a$ and set $(a)_{r'}=a/{(a)_r}$. If $r$ is odd and $(a,r)=1$, then $e(r,a)$ denotes the multiplicative order of $a$ modulo $r$. Define $e(2,a)$ to be $1$ if $4$ divides $a-1$ and to be $2$ if $4$ divides $a+1$. It is easy to see that $e(r,-a)=\nu(e(r,a))$, where

 \begin{equation}\label{e:nu}
 \nu(k)=\begin{cases} 2k& \text{if $k$ is odd},\\
                       k/2& \text{if } (k)_2=2,\\
                       k&\text{if } (k)_2>2.

\end{cases}
 \end{equation}

As usual, $\Phi_n(x)$ stands for the $n$th cyclotomic polynomial and $\varphi(x)$ for the Euler's totient function.

\begin{lemma}\label{l:r-part}
Let $a$ be an integer, $|a|\geq 2$, and let $r$ be an odd prime. For every integer $n\geq 1$, the following hold:
\begin{enumerate}
\item $r$ divides $\Phi_n(a)$ if and only if $(n)_{r'}=e(r,a);$
\item if $r^2$ divides $\Phi_n(a)$, then $n=e(r,a);$
\item if $r$  divides $a-1$, then $(a^n-1)_r=(n)_r(a-1)_r$.
\end{enumerate}
\end{lemma}
\begin{proof}
See, for example, \cite[Theorem 3.3]{74Her} for (i) and (ii), and \cite[Lemma 3.4]{74Her} for (iii).
\end{proof}

Fix an integer $a$ with $|a|>1$.  A prime $r$ is said to be a {\it primitive prime divisor} of $a^n-1$ if $e(r,a)=n$.  Observe that this definition of a primitive prime divisor is taken from \cite{05VasVd.t} and slightly differs  from the standard one, which requires the multiplicative order of $r$ modulo $a$ to be equal to $n$. The only difference is that $2$ is regarded as a primitive prime divisor of $a^2-1$ when $a\equiv 3\pmod 4$.

We write $R_n(a)$ to denote the set of all primitive divisors of $a^n-1$ and
$r_n(a)$ to denote an arbitrary element of $R_n(a)$, if such exists. So if we say that something holds for $r_n(a)$, we mean that this holds for all $r\in R_n(a)$. Note that $R_n(-a)=R_{\nu(n)}(a)$. The existence of primitive divisors for almost all pairs $(a,n)$ was proved by Bang~\cite{86Bang} and Zsigmondy~\cite{Zs}

\begin{lemma}[Bang--Zsigmondy]\label{l:zsigmondy}
Let $n\geq 1$ and $a$ be integers, $|a|\geq 2$. If the pair $(a,n)$ is not in $\{(2,1),(2,6),(-2,2),(-2,3),(3,1),(-3,2)\}$, then
the set $R_n(a)$ is not empty.
\end{lemma}

For $n\geq3$,  we denote the product of all primitive prime divisors of $a^n-1$ taken with multiplicities by $k_n(a)$. If $a$ is positive, then $k_n(a)$ is exactly $\Phi_n^*(a)$ introduced in \cite{74Her} and so it is calculated in \cite[Theorem 3.4]{74Her}.
We give a slightly different formula and, for convenience, extend it to negative $a$.

\begin{lemma}\label{l:kn}
Let $n\geq 3$ and $a$ be integers, $|a|\geq 2$. Then
\begin{equation}\label{eq:ki}
k_n(a)=\frac{\Phi_n(a)}{(r,\Phi_{(n)_{r'}}(a))},
\end{equation}
where $r$ is the largest prime dividing $n$. Moreover, $(r,\Phi_{(n)_{r'}}(a))=1$ unless $(n)_{r'}$ divides $r-1$.
\end{lemma}

\begin{proof}
Let $a$ be positive. For every $s\in R_n(a)$, it clear that $(a^n-1)_s=(\Phi_n(a))_s$ and so $k_n(a)=\Phi_n(a)/b$, where $b=\prod (\Phi_n(a))_s$ with $s$ runs over $\pi(\Phi_n(a))\setminus R_n(a)$. By \cite[Theorems 3.3 and 3.4]{74Her} if $s\in \pi(\Phi_n(a))\setminus R_n(a)$, then $s$ is the largest prime divisor of $n$. Furthermore, if $s$ is odd, then $(n)_{s'}=e(s,a)$ divides $s-1$ and $(\Phi_n(a))_s=s$.

If $s=2$, then $n=2^t\geq 4$ and $\Phi_n(a)=a^{n/2}+1$, so $b=(\Phi_n(a))_2=(2,a-1)$, as required. If $s>2$, then $b=(\Phi_n(a))_s=(s,\Phi_n(a))=(s,\Phi_{(n)_{s'}}(a))$.

Now we prove the claim for $k_n(-a)$. It is not hard to derive from the definition that $k_n(-a)=k_{\nu(n)}(a)$ (see, for example, \cite[Lemma 1.3]{15Vas}). Also  $n$ and $\nu(n)$ have the same largest prime divisor $r$. Hence $$k_{n}(-a)=k_{\nu(n)}(a)=\frac{\Phi_{\nu(n)}(a)}{(r,\Phi_{(\nu(n))_{r'}}(a))}.$$
We have $\Phi_{\nu(n)}(a)=\Phi_n(-a)$.
If $r=2$, then $\nu(n)=n$, $(n)_{r'}=1$ and $(2,\Phi_1(a))=(2,\Phi_1(-a))$. If $r>2$, then taking the $r'$-part and applying $\nu$ commute, so  $(r,\Phi_{(\nu(n))_{r'}}(a))=(r,\Phi_{(n)_{r'}}(-a))$, as required.
\end{proof}

\begin{lemma}\label{l:hering} Suppose that $a\geq 2$, $n\geq 2$ are integers, $n$ has exactly $k$ distinct prime factors and $m=\varphi(n)$. If $k$ is even, then $(a-1)a^{m-1}<\Phi_n(a)<a^m$. If $k$ is odd, then $a^{m}<\Phi_n(a)<a^{m+1}/(a-1)$.
\end{lemma}

\begin{proof}
This follows from \cite[Theorem 3.6]{74Her}.
\end{proof}

\begin{lemma}\label{l:uneq}
Let $n\geq 3$, $a,b\geq 2$ be integers and $\varepsilon,\tau\in\{+,-\}$. If $a>b$, then $\Phi_n(\varepsilon a)>\Phi_n(\tau b)$ unless $n\in\{3,6\}$, $\varepsilon\tau=-$ and $a=b+1$.
\end{lemma}

\begin{proof}
Let $m=\varphi(n)$. It follows from Lemma \ref{l:hering} that $\Phi_{n}(\varepsilon a)>(a-1)a^{m-1}$ and $\Phi_{n}(\tau b)<b^{m+1}/(b-1)$. If $m\geq 4$, then $$(a-1)a^{m-1}=(a-1)a^2a^{m-3}>b(b+1)^2b^{m-3}>b^{m+1}/(b-1)$$ since $(b+1)^2>b^3/(b-1)$. Similarly, if $a\geq b+2$, then $$(a-1)a^{m-1}\geq (b+1)(b+2)b^{m-2}>b^{m+1}/(b-1).$$ If $\tau=\varepsilon$, then $\Phi_n(\tau a)>\Phi_n(\tau b)$ as $\Phi_n(x)$ decreases on $(-\infty,-1)$ and increases on $(1, \infty)$.
\end{proof}

\begin{lemma}\label{l:tight} Suppose that $q$ and $u$ are coprime prime powers and $q$ is odd.
\begin{enumerate}
     \item Let $n\in\{10, 12,14\}$. If $k_{2n}(q)$ divides $k_{2n}(u)$, then $(u^n+1)/(2,u-1)$ does not divide $(q^n+1)/2$.
    \item Let $7\leq n\leq 17$ and $n$ odd. If $k_n(\varepsilon q)$ divides $k_n(\tau u)$ for some $\varepsilon,\tau\in\{+,-\}$, then $(u^n-\tau)/(4,u-\tau)$ does not divide $(q^n-\varepsilon)/2$.
    \item Let $7\leq n\leq 15$ and $n$ odd. If $k_{2n-2}(q)$ divides $k_n(\tau u)$
   and $k_{n}(\varepsilon q)$ divides $k_{2n-2}(u)$ for some $\varepsilon,\tau\in\{+,-\}$, then $(u^n-\tau)/(4,u-\tau)$ does not divide $(q^{n-1}+1)(q+\varepsilon)/(4,q-\varepsilon)$.
    \end{enumerate}
\end{lemma}

\begin{proof}

(i) Assume the opposite. Then $u<q$. If $n=12$ or $14$, then $k_{2n}(a)=\Phi_{2n}(a)$ for every integer $a\geq 2$ by Lemma \ref{l:kn}, and so $k_{2n}(u)<k_{2n}(q)$ by Lemma \ref{l:uneq}; a contradiction.

Let  $n=10$. Then $k_{2n}(q)=\Phi_{20}(q)/(5,q^2+1)$ and $k_{2n}(u)=\Phi_{20}(u)/(5,u^2+1)$. If $(5,q^2+1)\leq (5, u^2+1)$, then again $k_{2n}(u)<k_{2n}(q)$. Therefore, we can assume that $(5,q^2+1)=5$ and $(5,u^2+1)=1$. Then $(q^{10}+1)_5\geq25$ and $(u^{10}+1)_5=1$ by Lemma \ref{l:r-part}, and therefore $25(u^{10}+1)\leq q^{10}+1$, whence $25u^{10}<q^{10}$. Noting that $\Phi_{20}(x)=\Phi_{10}(x^2)$ and applying Lemma \ref{l:hering}, we see that $k_{20}(q)>8q^8/45$. Also it clear that $\Phi_{10}(x^2)<x^8$ and so $k_{20}(u)<u^8$. Hence $q^8<45u^8/8$. It follows that $25^4u^{40}<q^{40}<(45/8)^5u^{40}$, which yields $8^5\cdot 5^3<9^5$, a contradiction.

(ii) Suppose that $(u^n-\tau)/(4,u-\tau)$ divides $(q^n-\varepsilon)/2$. Then  $(u^n-1)/4\leq (q^n+1)/2$, so \begin{equation}\label{e:1} u^n\leq 2q^n+3.\end{equation}
Denote $m=\varphi(n)$. By Lemma \ref{l:hering}, we have $\Phi_n(\varepsilon q)>(q-1)q^{m-1}\geq 2q^m/3$ and $\Phi_n(\tau u)<u^{m+1}/(u-1)<2u^m$. Also note that prime factors of $k_n(\varepsilon q)$ are at least $2n+1$.

If $n=9$, then $3q^6<u^6$ by \cite[Lemma 2.5(iii)]{17Sta}, which yields $3q^9<u^9$, contradicting \eqref{e:1}. Let $n=15$. Since $k_{15}(a)=\Phi_{15}(a)$ and
$\Phi_{15}(\varepsilon q)\neq \Phi_{15}(\tau u)$ by Lemma \ref{l:uneq}, it follows that  $\Phi_{15}(\varepsilon q)\leq \Phi_{15}(\tau u)/31$. Then $2q^8/3<2u^8/31$ and therefore $10q^{15}<u^{15}$, a contradiction with~\eqref{e:1}.

Thus $n$ is prime and so $k_n(a)=\Phi_n(a)/(n,a-1)$. Suppose that  $(n,q-\varepsilon)\leq (n,u-\tau)$. Then $\Phi_n(\varepsilon q)$ divides $\Phi_n(\tau u)$. Lemma~\ref{l:uneq} implies that $\Phi_n(\varepsilon q)\neq\Phi_n(\tau u)$. Since prime factors of $\Phi_n(\tau u)$ are at least $n$, we have that $q^{n-1}<3u^{n-1}/n$ by Lemma~\ref{l:hering}. A direct calculation shows that $ (3/n)^{n/(n-1)}\leq (3/7)^{(7/6)}<2/5$ and so $u^n<2(3/n)^{n/(n-1)}u^n+3<4u^n/5+3$, whence $u^n<15$, a contradiction.

Let $(n,q-\varepsilon)=n$ and $(n,u-\tau)=1$. In this case $(q^n-\varepsilon)_n\geq n^2$ and $(u^n-\tau)_n=1$, so $n^2(u^n-\tau)/(4,u-\tau)$  divides $(q^n-\varepsilon)/2$, which yields $n^2u^n<2q^n+n^2+2$. On the other hand, $q\geq 2n-1$ and hence $k_n(\varepsilon q)>\frac{2n-2}{n(2n-1)}q^{n-1}$ by Lemma~\ref{l:hering}. Since $k_n(\tau u)<2u^{n-1}$, we get that $q^{n-1}<\alpha u^{n-1}$ for $\alpha=\frac{n(2n-1)}{n-1}$ and so $n^2u^{n}<2\alpha^{n/(n-1)}u^n+n^2+2$. It follows that $u^n<\frac{n^2+2}{n^2-2\alpha^{n/(n-1)}}<40$, where the last inequality is directly verified for $n=7,11,13,17$.

(iii) Again we assume the opposite. Then $u^n<2(q^{n-1}+1)(q+1)+1<3q^n$. In particular, if $q=3$, then $u=2$.

For every $n$, let $i$, $j$ and $d$ be as in Table \ref{tab:ij}. Then $\{n,2n-2\}=\{i,j\}$, $\varphi(i)>\varphi(j)$ and $k_i(\varepsilon q)\geq \Phi_i(\varepsilon q)/d$ in view of~(\ref{eq:ki}).
\begin{table}[!hbt]
   \centering
   \caption{}
   \label{tab:ij}
$\begin{array}{|c|c|c|c|c|}
\hline
n&i&j&d&{\varphi(i)-\varphi(j)}\\
\hline
7&7&12&(7,q^2-1)&2\\
9&16&9&2&2\\
11&11&20&(11,q^2-1)&2\\
13&13&24&(13,q^2-1)&4\\
15&28&15&1&4\\
\hline
\end{array}$
\end{table}

Suppose that $u\geq 7$. Then $q\geq 5$ and so $4q^{\varphi(i)}/5d<7u^{\varphi(j)}/6$ by Lemma~\ref{l:hering}. Together with  $u^n<3q^n$, this yields $$u^{\varphi(i)}<3^{\varphi(i)/n}q^{\varphi(i)}<3\cdot (35d/24)u^{\varphi(j)}<(35d/8)u^{\varphi(j)}.$$ Taking $\varphi(i)-\varphi(j)$ and $d$ from Table \ref{tab:ij}, we derive a contradiction since $35\cdot 11/8<7^2$ and $35\cdot 13/8<7^4$.

Suppose that $u\leq 5$. We noted above that if $q=3$, then $u=2$,
so since $u\neq q$ we have the inequality $u<q$. Since $2q^{\varphi(i)}/3d<2u^{\varphi(j)}$, we get that $q^{\varphi(i)}<3dq^{\varphi(j)}$. According to the table, this implies that $q\leq 5$. Then $d\leq 2$ and $q^2<6$, a contradiction.
\end{proof}

\subsection{Element orders and prime graphs of simple classical groups}

As was mentioned above, we denote simple classical groups according to the \emph{Atlas of Finite Groups}~\cite{85Atlas}. Along with this notation, we write $L_n^+(q)$ for $L_n(q)$ and $L_n^-(q)$ for $U_n(q)$. If $L$ is equal to $L_n^\pm(q)$, $S_{2n}(q)$, $O_{2n+1}(q)$, or $O_{2n}^\pm(q)$, then we say that $L$ is \emph{a group over a field of order $q$} (despite the fact that $U_n(q)$ originates from a subgroup of $SL_n(q^2)$) and refer to $n$ as $\prk(L)$. In other words, $\prk(L)$ is the dimension of $L$ if $L$ is linear or unitary and the Lie rank of $L$ otherwise.

Let $G$ be a finite group. The set $\omega(G)$ is uniquely determined by the subset $\mu(G)$ consisting of element orders that are maximal under divisibility relations.
Our standard references for the spectra of simple classical groups are \cite{08But.t} and \cite{10But.t} (with corrections from \cite[Lemma 2.3]{16Gr.t}).

\begin{lemma}{\em\cite[Lemma~1.3]{09VasGrMaz1.t}}\label{l:orders}
Suppose that $L$ is a simple classical group of Lie rank $l$ over a field of order $q$. Then element orders of $L$ do not exceed $\frac{q}{q-1}q^l$.
\end{lemma}

As usual, $\pi(G)$ is the set of all prime divisors of the order of $G$. The \emph{prime graph} $GK(G)$ of $G$ is the graph with vertex set $\pi(G)$ in which vertices $r$ and $s$ are adjacent if and only if $r\neq s$ and $rs\in\omega(G)$. A set $\rho\subseteq\pi(G)$ is a \emph{coclique} if the vertices of $\rho$ are pairwise nonadjacent.  We refer to a coclique in $GK(G)$ containing $r\in\pi(G)$ as an \emph{$\{r\}$-coclique}. We denote by $t(G)$ the greatest size of a coclique in $GK(G)$ and by $t(r,G)$ the greatest size of $\{r\}$-cocliques in $GK(G)$. A prime $r\in\pi(G)$ is called \emph{large} with respect to $G$ if $t(r,G)=t(G)$. Our standard references for adjacency in the prime graphs of simple groups are \cite{05VasVd.t} and \cite{11VasVd.t}.

In the remaining part of this subsection, $L$ is a classical group over a field of order $q$ and characteristic $p$. Of most importance for our purposes, are cocliques,  $\{2\}$-cocliques and $\{p\}$-cocliques of $GK(L)$ having maximal possible size. To deal adjacency in $GK(L)$,  we will use the function $\nu$ defined in  \eqref{e:nu} and the following functions $\eta$ and $\varphi(r,L)$ introduced in \cite{05VasVd.t} and \cite{15Vas}, respectively. If $k$ is a positive integer, then
$$\eta(k)=\begin{cases} k& \text{if $k$ is odd}\\
                       k/2& \text{if $k$ is even}\\
\end{cases}.$$

If $r\in\pi(L)\setminus\{p\}$, then $$\varphi(r,L)=\begin{cases} e(r,q) &\text{if } L=L_n(q),\\
                             \nu(e(r,q)) &\text{if } L=U_n(q),\\
                             \eta(e(r,q))& \text{otherwise}.
\end{cases}$$

\begin{lemma}\label{l:adj_criteria}
Let $\prk L=n\geq 4$ and  $r$, $s\in\pi(L)\setminus\{2,p\}$.
Put $k=e(r, q)$, $l=e(s, q)$ and suppose that $\varphi(r,L)\leq\varphi(s,L)$. Then the following statements hold.
\begin{enumerate}
\item  If $L=L^\varepsilon_n(q)$ and $\varphi(r,L)\geq 2$, then $r$ and $s$ are adjacent in $GK(L)$ if and only if $\varphi(r,L)+\varphi(s,L)\leq n$ or $\varphi(r,L)$ divides $\varphi(s,L)$.
\item  If $L\in\{O_{2n+1}(q), S_{2n}(q)\}$, then $r$ and $s$ are adjacent in $GK(L)$ if and only if $\varphi(r,L)+\varphi(s,L)\leq n$ or $\frac{l}{k}$ is an odd integer.
\item If $L=O^\varepsilon_{2n}(q)$, then $r$ and $s$ are adjacent in $GK(L)$ if and only if $2\varphi(r,L)+2\varphi(s,L)\leq 2n-(1-\varepsilon(-1)^{k+l})$, or $\frac{l}{k}$ is an odd integer, or $\varepsilon = +$, $n$ is even, $s\in R_n(q)$, $r\in R_{n/2}(q)$.
\end{enumerate}
\end{lemma}

\begin{proof}
This follows from \cite[Propositions 2.1 and 2.2]{05VasVd.t} and \cite[Propositions 2.4 and 2.5]{11VasVd.t}.
\end{proof}

\begin{lemma}\label{l:half}
Let $\prk(L)\geq 4$,  $r,s\in\pi(L)\setminus\{p\}$ and $r\neq s$. If $\varphi(r,L)$, $\varphi(s,L)\leq n/2$, then $rs\in\omega(L)$.
\end{lemma}

\begin{proof}
See \cite[Lemma~2.4(ii)]{15Vas}.
\end{proof}

Let $t(L)\geq 5$ and let $\rho$ be a  coclique of greatest size in $GK(L)$. Since $t(p, L)\leq 4$ by \cite[Table 4]{05VasVd.t} (also see Lemmas \ref{l:J(r,L)} and \ref{l:J(p,L)} below), it follows that $p\not\in\rho$ and so we can write $E(\rho, L)=\{e(r,q)\mid r\in \rho\}$.  We define $J(L)$ to be  the union of $E(\rho, L)$, where $\rho$ runs over all cocliques of greatest size in $GK(L)$, and $E(L)$ to be the intersection of these sets.

\begin{lemma}\label{l:tL}
Let $t(L)\geq 5$.
\begin{enumerate}
\item Let $\rho$ be a coclique of greatest size in~$GK(L)$. If $J(L)=E(L)$ then
$E(\rho,L)=E(L)$. If $J(L)\neq E(L)$ then $E(\rho,L)=E(L)\cup\{j\}$ for some $j\in J(L)\setminus
E(L)$. In particular, $|E(L)|\leq t(L)\leq |E(L)|+1$.
\item If $t(L)\geq 6$, then the sets $E(L)$, $J(L)\setminus E(L)$ and numbers $t(L)$ are as listed in Table~\emph{\ref{tab:tL}}.
\end{enumerate}
\end{lemma}
\begin{proof}
See \cite[Tables 2, 3]{11VasVd.t}.
\end{proof}

\begin{table}[!th]
\caption{Cocliques of greatest size in classical groups $L$ with $t(L)\geq 6$}\label{tab:tL}
\renewcommand{\arraystretch}{1.2}
\begin{center}
{\begin{tabular}{|c|l|c|c|c|}
  \hline
  $L$ & Conditions on $L$ & $t(L)$ & $E(L)$ & $J(L)\setminus E(L)$ \\
  \hline

  $L_n(q)$ & $n\geq11$ is odd and & $\frac{n+1}{2}$ & $\{i\mid \frac{n}{2}<i\leq n\}$ & $\varnothing$\\
           & $(n,q)\neq(11,2)$ &  & &\\
  & $n\geq12$ is even and & $\frac{n}{2}$ & $\{ {i}\mid \frac{n}{2}<i<n\}$ & $\{\frac{n}{2}, n\}$\\
  & $(n,q)\neq(12,2)$ &  & &\\
   & $n=12$ and $q=2$ & 6 & $\{7,8,9,10,11,12\}$ & $\varnothing$\\
  \hline
   $U_n(q)$ & $n\geq 11$ is odd & $\frac{n+1}{2}$ & $\{i\mid \frac{n}{2}< \nu\left(i\right)\leq n\}$ & $\varnothing$\\
  & $n\geq 12$ is even & $\frac{n}{2}$ & $\{ {i}\mid \frac{n}{2}< \nu\left(i\right)<n\}$ & $\{\frac{n}{2}, n\}$\\
  \hline
  $S_{2n}(q)$ or  & $n\geq 8$, $n\equiv{0}\pmod 4$& $\frac{3n+4}{4}$ &
  $\{i\mid \frac{n}{2}\leq\eta(i)\leq n\}$ &  $\varnothing$\\
  $O_{2n+1}(q)$ & $n\geq9$, $n\equiv{1}\pmod 4$ & $\frac{3n+5}{4}$ &
  $\{i\mid \frac{n}{2}<\eta(i)\leq n\}$ &  $\varnothing$\\
  &  $n\geq10$, $n\equiv{2}\pmod 4$ & $\frac{3n+2}{4}$ & $\{i\mid\frac{n}{2}<\eta(i)\leq n\}$ & $\{\frac{n}{2},n\}$\\
    & $n\geq7$, $n\equiv{3}\pmod 4$, and & $\frac{3n+3}{4}$ & $\{i\mid \frac{n+1}{2}<\eta(i)\leq n\}$ &  $\{ \frac{n-1}{2},n-1,$\\
 &  $(n,q)\neq(7,2)$  &  & & $n+1\}$\\
  & $n=7$ and $q=2$ & 6 & $\{5,7,10,12,14\}$ & $\{3,8\}$ \\
 \hline
 $O_{2n}^+(q)$ & $n\geq 8$, $n\equiv{0}\pmod 4$ & $\frac{3n}{4}$ & $\{i\mid
\frac{n}{2}\leq\eta(i)\leq n,$ & $\varnothing$\\
 & & & $i\neq2n\}$ & \\
 & $n\geq 9$, $n\equiv{1}\pmod 4$ & $\frac{3n+1}{4}$ & $\{i\mid
 \frac{n}{2}<\eta(i)\leq n,$ & $\{n-1, n+1\}$\\
 &  &  & $i\neq2n,n+1\}$ &\\
 & $n\geq10$, $n\equiv{2}\pmod 4$ & $\frac{3n-2}{4}$ & $\{i\mid
\frac{n}{2}<\eta(i)\leq n,$ & $\{\frac{n}{2},n\}$\\
 &  &  & $i\neq2n\}$ &\\
 & $n\geq7$, $n\equiv{3}\pmod 4$ & $\frac{3n+3}{4}$ & $\{i\mid
\frac{n-1}{2}\leq\eta(i)\leq n,$ & $\varnothing$\\
 & & & $i\neq2n,n-1\}$ &\\
 \hline
 $O_{2n}^-(q)$& $n\geq8$, $n\equiv{0}\pmod 4$ & $\frac{3n+4}{4}$ & $\{i\mid
\frac{n}{2}\leq\eta(i)\leq n\}$ & $\varnothing$\\
 & $n\geq 9$, $n\equiv{1}\pmod 4$ & $\frac{3n+1}{4}$ & $\{i\mid
\frac{n}{2}<\eta(i)\leq n,$
 & $\{\frac{n+1}{2}, n-1\}$\\
 & & & $i\neq n,\frac{n+1}{2}\}$ & \\
 & $n\geq10$, $n\equiv{2}\pmod 4$ & $\frac{3n+2}{4}$ & $\{i\mid
\frac{n}{2}<\eta(i)\leq n\}$ &
 $\{\frac{n}{2}, n-2, n\}$\\
 & $n\geq 7$, $n\equiv{3}\pmod 4$, and & $\frac{3n+3}{4}$ & $\{i\mid
\frac{n-1}{2}\leq\eta(i)\leq n,$ & $\varnothing$\\
 & $(n,q)\neq(7,2)$ &  &$i\neq n,\frac{n-1}{2}\}$ &\\

 \hline
\end{tabular}}
\end{center}
\end{table}

Proving Theorem \ref{t:main}, we will deal with simple classical groups $L$ and $S$ satisfying $6\leq t(L)\leq 13$ and $t(L)\leq t(S)\leq t(L)+1$ (see Lemma \ref{l:rest} below). So it is convenient to have a list of simple classical groups $L$ with $6\leq t(L)\leq 14$, and  these groups are given in  Table \ref{tab:6tL14}.

\begin{table}[!hbt]
   \centering
   \caption{The classical groups $L$ with $6\leq t(L)\leq 14$} \label{tab:6tL14}
   \renewcommand{\arraystretch}{1.2}
   \begin{tabular}{|c|c|}
\hline
$t(L)$ &  Type of $L$  \\ \hline
6 & $L^{\pm}_{11}$  (except $L_{11}(2)$), $L^{\pm}_{12}$, $O_{15}$, $S_{14}$, $O_{14}^+$, $O_{16}^+$, $O_{14}^-$ (except $O_{14}^-(2)$) \\ 
7 & $L^{\pm}_{13}$, $L^{\pm}_{14}$, $O_{17}$, $S_{16}$, $O_{18}^{\pm}$, $O_{20}^+$, $O_{16}^-$  \\ 
8 & $L^{\pm}_{15}$, $L^{\pm}_{16}$, $O_{19}$, $S_{18}$, $B_{21}$, $S_{20}$, $O_{20}^-$ \\ 
9 & $L^{\pm}_{17}$, $L^{\pm}_{18}$, $O_{23}$, $S_{22}$, $O_{22}^\pm$, $O_{24}^+$  \\ 
10 & $L^{\pm}_{19}$, $L^{\pm}_{20}$,  $O_{25}$, $S_{24}$, $O_{24}^-$, $O_{26}^\pm$, $O_{28}^+$ \\ 
11 & $L^{\pm}_{21}$, $L^{\pm}_{22}$, $O_{27}$, $S_{26}$, $O_{29}$, $S_{28}$, $O_{28}^-$  \\ 
12 & $L^{\pm}_{23}$, $L^{\pm}_{24}$,  $O_{31}$, $S_{30}$, $O_{30}^\pm$, $O_{32}^+$ \\ 
13 & $L^{\pm}_{25}$, $L^{\pm}_{26}$, $O_{33}$, $S_{32}$, $O_{32}^-$, $O_{34}^\pm$, $O_{36}^+$  \\ 
14 & $L^{\pm}_{27}$, $L^{\pm}_{28}$, $O_{35}$, $S_{34}$, $O_{37}$, $S_{36}$, $O_{36}^-$   \\ \hline
\end{tabular}
\end{table}

Comparing to cocliques of greatest size, $\{2\}$-cocliques and $\{p\}$-cocliques of greatest size have a more rigid structure.

\begin{lemma} \label{l:J(r,L)} Suppose that $\prk(L) \geq 4$ and $r\in\{2,p\}$. If $s\neq r$ and $rs\not\in\omega(L)$, then $s\not\in\{2,p\}$. If $J(r,L)=\{e(s,q)\mid s\in\pi(L), s\neq r, rs\not\in\omega(L)\}$, then $|J(r,L)|=t(r,L)-1$ and  every $\{r\}$-coclique $\rho$ of $GK(L)$ of greatest size has form $\{r\}\cup \bigcup\{r_i(q)\}$, where $i$ runs over $J(r,L)$.
\end{lemma}

\begin{lemma} \label{l:J(p,L)} Suppose that $\prk(L) \geq 8$
if $L$ is linear or unitary and $\prk(L) \geq 5$ otherwise. Then $\varnothing \neq J(2,L)\subseteq J(p,L)$ and $J(p,L)$ is as follows.
\begin{enumerate}
    \item If $L=L_n(q)$ or $U_n(q)$, then $J(p,L)=\{n,n-1\}$ or $\{\nu(n),\nu(n-1)\}$, respectively.
    \item If $L=S_{2n}(q)$ or $O_{2n+1}(q)$, where $n$ is even, then $J(p,L)=\{2n\}$.
    \item If $L=S_{2n}(q)$, $O_{2n+1}(q)$, or $O_{2n+2}^+(q)$, where $n$ is odd, then $J(p,L)=\{n,2n\}$.
    \item If $L=O_{2n}^+(q)$ or $O_{2n}^-(q)$, where $n$ is odd, then $J(p,L)=\{\varepsilon n,2n-2\}$, where $-n$ stands for $2n$.
      \item If $L=O_{2n}^-(q)$, where $n$ is even, then  $J(p,L)=\{2n,2n-2, n-1\}$ and $2n\in J(2,L)$.
\end{enumerate}
\end{lemma}

It is convenient to consider the group $O_{2n}^-(q)$, where $n$ is even, together with the groups $S_{2n}(q)$ and $O_{2n+1}(q)$, so we define $J'(p,L)=J'(2,L)=\{2n\}$ if $L=O_{2n}^-(q)$ and $n$ is even, and   $J'(2,L)=J(2,L)$ and $J'(p,L)=J(p,L)$ otherwise.

\begin{lemma}\label{l:disjoint} Suppose that $\prk(L) \geq 8$
if $L$ is linear or unitary and $\prk(L) \geq 5$ otherwise.
\begin{enumerate}
\item If $J(p,L)=\{i,j\}$ with $i\neq j$, then $r_i(q)$ and $r_j(q)$ have disjoint neighborhoods in $GK(L)$.
\item If $L=S_{2n}(q)$, $O_{2n+1}(q)$, or $O_{2n}^-(q)$, where $n$ is even, and $j\in\{n-1,2n-2\}$, then $r_{2n}(q)$ and $r_j(q)$ have disjoint neighborhoods in $GK(L)$.
\item  For every $i\in J'(p,L)$, there is a unique $m_i(q)\in\mu(L)$ divisible by $r_i(q)$ and $m_i(q)$ is as in Table {\rm \ref{tab:mi}}.
\end{enumerate}
\end{lemma}

\begin{table}[!h]
   \centering
   \caption{The numbers $m_i(q)$ for $i\in J'(p,L)$} \label{tab:mi}
\renewcommand{\arraystretch}{1.4}
$
\begin{array}{|l|r@{=}lr@{=}l|}
\hline
L&\multicolumn{4}{|c|}{m_i(q)}\\
\hline
L_n(q), n\geq 9& m_n(q)&\frac{q^n-1}{(q-1)(n,q-1)}& m_{n-1}(q)&\frac{q^{n-1}-1}{(n,q-1)}\\
U_n(q), n\geq 9& m_{\nu(n)}(q)&\frac{q^n-(-1)^n}{(q+1)(n,q+1)}& m_{\nu(n-1)}(q)&\frac{q^{n-1}-(-1)^{n-1}}{(n,q+1)}\\
S_{2n}(q), O_{2n+1}(q), O_{2n+2}^+(q), n\geq 5 \text{ odd}&m_{2n}(q)&\frac{q^n+1}{(2,q-1)}& m_{n}(q)&\frac{q^n-1}{(2,q-1)}\\
S_{2n}(q), O_{2n+1}(q), O_{2n}^-(q), n\geq 6 \text{ even}& m_{2n}(q)&\frac{q^n+1}{(2,q-1)}&\multicolumn{2}{c|}{}\\
O_{2n}^\varepsilon(q),   n\geq 5 \text{ odd} & m_{2n-2}(q)&\frac{(q^{n-1}+1)(q+\varepsilon)}{(4,q-\varepsilon)}& m_{\varepsilon n}(q)&\frac{q^n-\varepsilon}{(4,q-\varepsilon)}\\
\hline
\end{array}
$
\end{table}


Lemma \ref{l:disjoint} provides some vertices of $GK(L)$ with disjoint neighborhoods but there can be another pairs of vertices enjoying this property.
The next lemmas gives some properties of such pairs.

\begin{lemma}[{\cite[Lemma 2.8]{26GSV_arxiv}}]\label{l:unique}
Let $S$ be a simple classical group over a field of order $u$ and $t(S)\geq 5$. Assume that $u>3$ in all items except {\rm(v)}. Suppose that $s_1\in\pi(S)$ and $s_2\in \pi(S)$ have disjoint neighborhoods in $GK(S)$. Then, up to renumbering of $s_1$ and $s_2$, the following hold.

\begin{enumerate}
\item If $S=L_m^\tau(u)$, where $m$ is even, then  $s_1\in R_{m-1}(\tau u)$. If, in addition, $u-\tau$ does not divide $m$, then $s_2\in R_m(\tau u)$.

\item If $S=L_m^\tau(u)$, where $m$ is odd, then $s_1\in R_m(\tau u)$. If, in addition, $3$ divides $m$,  then $s_2\in R_{m-1}(\tau u)\cup R_{m-2}(\tau u)$.

\item If $S$ is one of $S_{2m}(u)$, $O_{2m+1}(u)$, or $O_{2m+2}^+(u)$,  where $m$ is odd, then $s_1\in R_{2m}(u)$ and $s_2\in R_{m}(u)$.

\item If $S=O_{2m}^\tau(u)$, where $m$ is odd, then  $s_1\in R_m(\tau u)$. In, in addition, $\tau u\neq +5$,  then $s_2\in R_{2(m-1)}(\tau u)$.

\item If $S$ is one of
$S_{2m}(u)$, $O_{2m+1}(u)$,  or $O_{2m}^-(u)$, where
$m$ is even, then $s_1\in R_{2m}(u)$. If, in addition, $(m)_2=2$, then $s_2\in R_{m-1}(\tau u)\cup R_{2(m-1)}(u)$ if $S\neq O_{2m}^-(u)$ and $s_2\in R_{m-1}(\tau u)\cup R_{2(m-1)}(u)\cup R_{2(m-2)}(u)$ otherwise.
\end{enumerate}
\end{lemma}

\subsection{The structure of groups isospectral to simple classical groups}

\begin{lemma}\cite[Lemma~3.1]{26GSV_arxiv}\label{l:reduction}
Let  $L$ be a simple classical group over a field of odd characteristic $p$, $t(L)\geq 5$ and $\omega(G)=\omega(L)$. If $G$ is not an almost simple group with socle isomorphic to $L$, then $S\leq \overline G=G/K\leq \Aut S$, where $K$ is the solvable radical of $G$ and $S$ is a simple classical group in characteristic $v\neq p$. Furthermore, $K$ is nilpotent and the following hold:
\begin{enumerate}
     \item if $\rho$ is a coclique in $GK(G)$ with $|\rho|\geq3$, then at most one prime in $\rho$ divides $|K|\cdot|\overline{G}/S|$;
    \item if $r\in\pi(G)$ is not adjacent to $2$ in $GK(G)$, then $r$ does not divide  $|K|\cdot|\overline{G}/S|$.
\end{enumerate}
\end{lemma}

\begin{lemma}\label{l:rest}
Let $L$ and $S$ be as in Lemma {\rm \ref{l:reduction}}.  Then $6\leq t(L)\leq 13$, $L$ is symplectic or orthogonal, $GK(L)$ is connected and $0\leq t(S)-t(L)\leq 1$.
\end{lemma}

\begin{proof} We have $t(L)\leq 13$ by \cite[Theorem 1.2]{17Sta}.
The fact that $L$ cannot be linear or unitary follows from \cite[Theorem 1]{26GSV_arxiv}, so $L$ is symplectic or orthogonal. In this case, if $GK(L)$ is disconnected, then it is known that $S\simeq L$
(see \cite[Theorem 2.1]{23Survey}). 
The inequalities $0\leq t(S)-t(L)\leq 1$ hold by \cite[Proposition 5.1]{26GSV_arxiv}.
\end{proof}

\begin{lemma}\label{l:primes}
Let $L$, $\overline G$, $K$, $S$, $p$ and $v$ be as in Lemma {\rm \ref{l:reduction}}. Let $L$ and $S$ be groups over fields of order $q$ and $u$, respectively.
\begin{enumerate}
\item If $r\in\pi(K)$, then $t(r,L)\leq 3$ and $r$ divides $v(q^2-1)$.
\item If $r$ does not divide $p(q^4-1)$ and $r\in\pi(S)$, then $t(r, S)\geq t(r, L)$.
\item If $r\in\pi(L)$ is large with respect to $L$, then $r$ is coprime to $v\cdot|K|\cdot |\overline G/S|$, $r\in\pi(S)$ and  $t(r,S)\geq t(L)$.
\item If $\rho$ is a coclique in $GK(L)$ of size at least $2$ and $\rho\subseteq\pi(S)\setminus\{v\}$, then at most one prime in $\rho$ is not large with respect to $S$.
\item If $t(S)=t(L)+1$, then $2u^3<q^2$.
\end{enumerate}

\end{lemma}

\begin{proof} By Lemma \ref{l:rest} it follows that $6\leq t(L)\leq 13$ and $L$ is symplectic or orthogonal. Now (i) is \cite[Lemma 3.3]{26GSV_arxiv}, while (ii) and (iii) is \cite[Lemma 3.5]{26GSV_arxiv}.

(iv) Note that $t(S)\geq 6$ by Lemma \ref{l:rest} and so $m=\prk(S)\geq 7$, see Table \ref{tab:6tL14}.  By Lemma \ref{l:half}, if $s,r\in\pi(S)\setminus \{v\}$ and $\varphi(r,S)$, $\varphi(s,S)\leq m/2$, then $rs\in\omega(S)$. So there is at most one prime $r\in \rho$ with $t(r,S)\leq m/2$. Now it is easy to derive from Table \ref{tab:tL} that every $t\in \pi(S)\setminus\{v\}$ with $\varphi(t, S)>m/2$ is large with respect to $S$.

(v)  By \cite[Lemmas 4.6 and 4.10]{26Sta.t}, there is $j\in J(S)$ such that $k_j(u)$ divides $p(q^2-1)\log_vu$. Furthermore, if $p$ divides $k_j(u)$, then we may assume that either $p<31$ or $p$ divides $\log_vu$.
Since $t(S)\geq 7$, it follows from Table \ref{tab:tL} that $j=5$ or $j\geq 7$. Now $2u^3<q^2$ by \cite[Lemma 2.9]{26Sta.t}.
\end{proof}

\begin{lemma}{\em\cite[Lemma~3.5]{15Vas}\label{l:3.5}} Let $L$ be a simple classical group over a field of order~$q$ and characteristic~$p$,
$r\in\pi(L)$, $r^s\in\omega(P)$, where $P$ is a proper parabolic subgroup of $L$, and
$(r,6p(q+1))=1$. If $L$ acts faithfully on a vector space $V$ over the field of characteristic $t$ distinct from~$p$, then $tr^{s}\in\omega(V\rtimes L)$.
\end{lemma}

\section{Proof of Theorem \ref{t:main}: general considerations}\label{s:structure}

Let $L$ be a simple symplectic or orthogonal group  over a field of order $q$ and characteristic $p\neq 2$ such that $6\leq t(L)\leq 13$ and let $G$ be a finite group with $\omega(G)=\omega(L)$. Suppose that $G$ is not an almost simple group with socle isomorphic to $L$. By Lemmas \ref{l:reduction} and \ref{l:rest},  we have $S\leq \overline G=G/K\leq \Aut S$, where $K$ is the solvable radical of $G$ and $S$ is a simple classical group over a field of order $u$ and characteristic $v\neq p$ with $t(S)\in\{t(L), t(L)+1\}$. Also the prime graph of $GK(L)$ is connected, in particular, $L$ is different from $S_{2n}(q)$, $O_{2n+1}(q)$, $O_{2n}^-(q)$, where $n=8,16$.

Let $i\in J(L)$ and $r\in R_i(q)$. In view of Lemma \ref{l:primes}(iii), we know  that $r\in\pi(S)$ and $k_i(q)\in\omega(S)$ but it is not necessary true that $e(r, u)$ is uniquely determined by $i$, even if $r$ is large with respect to $S$. To handle this, we introduce the following notation. If $i=\{i_1,i_2\}$, where $i_1,i_2\geq 3$ are distinct integers, and $a$ is an integer with $|a|>1$, then set $R_{i}(a)=R_{i_1}(a)\cup R_{i_2}(a)$ and  $k_i(a)=k_{i_1}(a)k_{i_2}(a)$.

\begin{lemma}\label{l:trans}
Let $i\in E(L)$ or $i=\{i_1,i_2\}$, where  $i_1,i_2\in J(L)\setminus E(L)$. Suppose that every $r\in R_i(q)$ is large with respect to $S$. Then
$R_i(q)\subseteq R_j(u)$ and $k_i(q)$ divides $k_j(u)$, where either $j\in E(S)$ or $j=\{j_1,j_2\}$ for some  $j_1,j_2\in J(S)\setminus E(S)$.
\end{lemma}

\begin{proof}
Let $r\in R_i(q)$ and $j=e(r,u)$. If $j\in E(S)$, then $e(s,u)=j$   for every $s\in R_i(q)$ since elements of $R_j(u)$ are not adjacent to $r_l(u)$ with $l\in J(S)\setminus \{j\}$, and so $R_i(q)\subseteq R_j(u)$ and $k_i(q)$ divides $k_j(u)$. According to Table~\ref{tab:tL},
we see that $|J(S)\setminus E(S)|\leq 3$.
Suppose that $j\in J(S)\setminus E(S)$. Then $e(s,u)\in J(S)\setminus E(S)$ for every $s\in R_i(q)$ since elements of $R_j(u)$ are not adjacent to $r_l(u)$ with $l\in E(S)$. If $\{e(s,u)\mid s\in R_i(q)\}=\{j_1,j_2,j_3\}$, then $r_{j_1}r_{j_2}r_{j_3}\in \omega(S)$ for some $r_{j_l}\in R_{j_l}(u)$. Using the description of $\omega(S)$, it is not hard to verify that this is not the case.
\end{proof}

\begin{lemma}\label{l:R8}
$R_8(q)\cap (\pi(\overline G/S)\cup\{v\})=\varnothing$.
\end{lemma}

\begin{proof}

First, assume that $v\in R_8(q)$. Then $r_8(q)$ is small with respect to $L$ by Lemma~\ref{l:primes}(iii), and hence $n=\prk L\geq 9$ according to Table~\ref{tab:tL}. Now it is easy to see that $r_8(q)$ is not adjacent in $GK(L)$ to at least four primes $r_j(q)$ that are large with respect to $L$ and have $\eta(j)\in\{n,n-1,n-2,n-3\}$. It follows from Lemma~\ref{l:primes}(ii) that $t(v,S)\geq 5$, a contradiction with Lemma~\ref{l:J(p,L)}.

Assume that $r\in R_8(q)\cap \pi(\overline G/S)$. Suppose that there exist integers $i$ and $j$ such that $\{r_i(q), r_j(q), r_8(q)\}$
is a coclique of size $3$ in $GK(L)$. By \cite[Lemma 4.3]{26Sta.t}, it follows that $r$ divides either $k_i(q)-1$ or $k_j(q)-1$. In particular, we will get an immediate contradiction if we find $i$ and $j$ such that $(k_8(q), k_i(q)-1)=(k_8(q), k_j(q)-1)=1$.

If $16\leq n\leq 18$, then we can take $i=15$ and $j=30$, while if $n=15$, then we can take $i=28$ and $j=15$ or $30$ depending on $L$. Indeed,
we see that $\{r, r_i(q), r_j(q)\}$ is a coclique of size 3 in $GK(L)$ by Lemma \ref{l:adj_criteria}.
Using Lemma~\ref{l:kn}, we find that $k_{i}(q)=\Phi_i(q)$ for $i=15,30,28$. Also $\Phi_{15}(x)-1=x(x^4-1)(x^3-x^2+1)$, and therefore $(\Phi_{15}(q)-1,q^4+1)=2(q^3-q^2+1,q^4+1)=2(q^3-q^2+1,q-1)=2$, which yields $(\Phi_{15}(q)-1,(q^4+1)/2)=(\Phi_{30}(q)-1,(q^4+1)/2)=1$. Similarly, we see that $\Phi_{28}(x)-1=x^2(x^{12}-1)/(x^2+1)$ and $(q^{12}-1,q^4+1)=2$. So $(\Phi_{28}(q)-1,k_8(q))=1$. Therefore, we can assume that $n\leq 14$.

If $i$ is an odd prime and $(q-\epsilon, i)=1$, then $k_{i^l}(\epsilon q)-1=\Phi_{i^l}(\epsilon{q})-1$ divides $q^{i^{l-1}}(q^{i^{l-1}(i-1)}-1)$ and hence $k_{i^l}(\epsilon q)-1$ is coprime to $k_8(q)$ whenever $8$ does not divide $i-1$.

Let $n=13, 14$. We choose $\epsilon, \tau\in \{+,-\}$ such that $(q-\varepsilon, 11)=(q-\tau, 13)=1$. Then $\{r_{11}(\epsilon q), r_{13}(\tau q), r_8(q)\}$ is a coclique in $GK(L)$, and so we are done by the above argument, unless $L=O_{26}^\varepsilon(q)$ and $(13, q-\varepsilon)=13$. In this case we take $k_i(q)$, $k_j(q)$ to be first $k_{13}(\varepsilon q), k_{11}(q)$, then $k_{13}(\varepsilon q), k_{22}(q)$, and finally $k_{22}(q), k_{11}(q)$.
We may assume that $r$ does not divide $k_{11}(\epsilon q)-1$ and, therefore,  $r$ divides both $k_{13}(\varepsilon q)-1=(\Phi_{13}(\varepsilon q)-13)/13$ and $k_{11}(-\epsilon q)-1=(\Phi_{11}(\epsilon q)-11)/11$. Using the extended Euclidean algorithm for polynomials over rational numbers, we check that $(\Phi_{13}(\pm q)-13, (q^4+1)/2)$ divides $3697$ and $(\Phi_{11}(\pm q)-11, (q^4+1)/2)$ divides $569$. So $r$ divides  $(3697,569)=1$, a contradiction.

If $n=11,12$, then we repeat the above argument with $9$ and $11$ in place of $11$ and $13$, which left us with the case when $L=O_{22}^\varepsilon(q)$,  $(11, q-\varepsilon)=11$ and $r$ divides both $\Phi_{11}(\varepsilon q)-11$ and $\Phi_9(\epsilon q)-3$ for some $\epsilon\in\{+,-\}$. Since $(\Phi_{11}(\pm q)-11, (q^4+1)/2)$ divides $569$ and $(\Phi_9(\pm q)-3, (q^4+1)/2)$ divides 17, this is a contradiction.

If $n=9,10$ or $n=7,8$, then we argue in a similar manner with $7$ and $9$, or $5$ and $7$, respectively. It remains to observe that $(\Phi_7(\pm q)-7, k_8(q))$ divides 1201 and $(\Phi_5(\pm q)-5, k_8(q))$ divides 97.
\end{proof}

\begin{lemma}\label{l:lm_even} Let $S=L_m^\tau(u)$, where $12\leq m\leq 28$ is even. Suppose that $R_j(q)\not\subseteq R_m(\tau u)$ for any $j\in J'(p,L)$ and let $i\in J'(2,L)$. Then $u-\tau$ divides $m$, $m-1$ is not prime, $m_i(q)=(u^{m-1}-\tau)/(u-\tau)$ is not divisible by $3$, and $k_i(q)=k_{m-1}(\tau u)$.
\end{lemma}

\begin{proof}
By Lemma \ref{l:reduction}(ii), it follows that $e(r_i(q),u)\in J(2,S)\subseteq J(v,S)=\{\tau m,\tau(m-1)\}$, where by $-k$ we mean $\nu(k)$. Since $k_i(q)\in\omega(S)$ and $k_i(q)$ does not divide $k_m(\tau u)$ by assumption,  $k_i(q)$ divides $k_{m-1}(\tau u)$. Then $(u^{m-1}-\tau)/(m,u-\tau)$ divides $m_i(q)$ by Lemma \ref{l:disjoint}(iii).

If $|J(p,L)|=2$, then let $j\in J(p,L)\setminus \{i\}$. If $|J(p,L)|\neq 2$, then there is $j\in\{n-1, 2n-2\}$, such that $r_j\not\in R_m(\tau u)$. So by Lemma \ref{l:disjoint}, in either case, we have $j\in J(L)$ such that $r_i(q)$ and $r_j=r_j(q)$ have disjoint neighborhoods in $GK(S)$ and $r_j\in R_k(\tau u)$ for some $k<m-1$ (recall that $r_j\neq v$ by Lemma \ref{l:primes}).

By Lemma \ref{l:unique}(i), it follows that $u=2,3$ or $u-\tau$ divides $m$. In particular, it is easy to see that $\log_vu \leq 3$. Since $\overline G\cap \Inndiag S=1$ by \cite[Theorem 1]{22GreVas}, it follows that $\pi(\overline G/S)\subseteq\{2,3\}$.

We claim that $\pi(m_i(q))\cap (\pi(\overline G/S)\cup\pi(K))=\varnothing$. As, by definition, $m_i(q)$ is odd, it sufficient to show that $r_j$ is adjacent to every $r\in \pi(K)\cup\{3\}$ in $GK(G)$. Recall that $r_j\in R_k(\tau u)$ for some $k<m-1$. This implies that $r_j$ is adjacent to both $v$ and $3$ in $GK(S)$ and that $r_j$ divides the order of a proper parabolic subgroup of $S$.  Also observe that $k\geq 3$ since otherwise $t(r_j,S)\leq 3$ but $t(r_j,S)\geq t(r_j,L)=t(L)\geq 6$ by Lemma \ref{l:primes}. By nilpotency of $K$ and Lemma~\ref{l:3.5}, it follows that $rr_j\in\omega(G)$ for every $r\in\pi(K)\setminus\{v\}$.

Thus $m_i(q)\in\omega(S)$. Since $(u^{m-1}-\tau)/(m,u-\tau)$ divides $m_i(q)$ by Lemma \ref{l:disjoint} and lies in $\mu(S)$ by \cite[Lemma 4]{08Zav1.t}, we see that $m_i(q)=(u^{m-1}-\tau)/(m,u-\tau)$. Assume that $u-\tau$ does not divide $m$. Then $u=2$ or $u=3$, and $\tau=-$. It follows that $m_i(q)=2^{m-1}+1$ is divisible by $3$,  or $m_i(q)=(3^{m-1}+1)/2$ is even, a contradiction. Hence $u-\tau$ always divides $m$. In particular, $(m-1,u-\tau)=1$ and $m_i(q)=(u^{m-1}-\tau)/(u-\tau)$.

Since $GK(L)$ is connected, there is $r\in \pi(m_i(q))\setminus R_i(q)$. By the form of $m_i(q)$, we see that $\varphi(r,L)\leq n/3$ and so $r$ is small with respect to $L$. Hence there is $l\in J(L)\setminus\{i\}$ such that $rr_l(q)\in\omega(L)$. Then $rr_l(q)\in\omega(S)$, which implies that $r\not\in k_{m-1}(\tau u)$. Thus $k_i(q)=k_{m-1}(\tau u)$ and $m_i(q)/k_i(q)=(u^{m-1}-\tau)/((u-\tau)k_{m-1}(\tau u))$. If $m-1$ is prime, then $k_{m-1}(\tau u)=(u^{m-1}-\tau)/(u-\tau)$, and therefore  $m_i(q)/k_i(q)=1$, which is a contradiction.
\end{proof}


\begin{lemma}\label{l:lm_odd}
Let  $S=L_m^\tau(u)$, where $m$ is odd, and $t(S)=t(L)+1$. Suppose that there is $k\in J(L)$ and $r\in R_k(q)$ such that $r$ is not large with respect to $S$.
Then $r\in R_{m-1}(\tau u)$ and every $s\in R_i(q)$, where $i\in J(L)\setminus\{k\}$, is large with respect to $S$.
\end{lemma}

\begin{proof}
See \cite[Lemma 3.7]{26GSV_arxiv}.
\end{proof}

\begin{lemma}\label{l:s2n_even} Let $L$ be one of the groups $S_{2n}(q)$, $O_{2n+1}(q)$, or $O_{2n}^-(q)$, where $n$ is even.
\begin{enumerate}
    \item If $S$ is one of $S_{2m}(u)$, $O_{2m+1}(u)$, $O_{2m+2}^+(u)$, $O_{2m}^\tau (u)$,  where $m$ is odd, then $u=2,3$ or $S=O_{2m}^+(5)$.
     \item  If $S$ is one of $S_{2m}(u)$, $O_{2m+1}(u)$, or $O_{2m}^-(u)$, where $m$ is even, then $k_{2n}(q)$ divides $k_{2m}(u)$.
     \item  If $S=L_m^\tau(u)$, where $m$ is odd and $u>3$, then $k_{2n}(q)$ divides $k_{m}(\tau u)$.
\end{enumerate}
\end{lemma}

\begin{proof}
See \cite[Lemma 3.8]{26GSV_arxiv}.
\end{proof}

\section{Proof of Theorem \ref{t:main}: Case $t(S)=t(L)+1$}

We keep the notation of the previous section. The general idea of the proof of the theorem is to find some inconsistent system of the form $f_1(u)<f_2(q)$, $f_3(q)<f_4(u)$. In this section we deal with the case $t(S)=t(L)+1$ and our first inequality is provided by Lemma \ref{l:primes}(v):
\begin{equation}\label{e:dif1}
2u^3<q^2.
\end{equation}

In particular, $q\geq 5$ because $u\geq 2$.

We will use the following conventions. First, since the prime graphs of $S_{2n}(q)$ and $O_{2n+1}(q)$ are the same, we omit $O_{2n+1}(q)$ in the below tables for brevity. Second, if $i$ is odd, then regarding elements of $J(L)$ or $J(p,L)$,  we sometimes write $-i$ instead of $2i$. So, for example, if $L=S_{2n}(q)$, where $n$ is odd, then we can say that $J(p,L)$ consists of $n$ and $-n$.

\begin{lemma}\label{l:t(S)=t(L)+1so}
If $S$ is symplectic or orthogonal, then $t(S)\neq t(L)+1$.
\end{lemma}

\begin{proof}
Assume the opposite. Suppose that $S\neq O_{18}^\pm(u), O_{20}^+(u)$ and let $l$, $i$ and $d$ be as in Table \ref{tab:t(L)+1so}.
\begin{table}[h]
\caption{Indices for Lemma \ref{l:t(S)=t(L)+1so}}\label{tab:t(L)+1so}
$\begin{array}{|c|c|c|c|c|c|c|}
\hline
t(L)&l&\text{type of }L&i&d&3\varphi(i)/2-l&\frac{4d}{2^{\varphi(i)/2}}\rule{0pt}{12pt} \\[2pt]
\hline
6&8&S_{14}, O_{16}^+, O_{14}^\pm&7&7&1&7/2\\
7&10& O_{20}^+, O_{18}^\pm&16&1&2&1/4\\
8&12&S_{18}, S_{20}, O_{20}^-&16&1&0&1/4\\
9,10&14&S_{22}, O_{22}^\pm, S_{24}, O_{24}^\pm, O_{26}^\pm, O_{28}^+&11&11&1&11/8\\
11&16&S_{26}, S_{28}, O_{28}^-&13&13&2&13/16\\
12,13&18&S_{30}, O_{30}^\pm, O_{32}^+, O_{34}^\pm, O_{36}^+ &13&13&0&13/16\\
\hline
\end{array}$
\end{table}

It follows from Table \ref{tab:6tL14} that for every fixed $t(L)$,
the Lie rank of $S$ is at most $l$ and the type of $L$ is as stated. By Table \ref{tab:tL}, we see that  $\epsilon i\in J(L)$ for some suitable $\epsilon\in\{+,-\}$ and so $k_i(\epsilon q)\in\omega(S)$. Since $k_i(\epsilon q)>q^{\varphi(i)}/2d$ by Lemmas \ref{l:kn} and \ref{l:hering}, we have $q^{\varphi(i)}/2d<2u^l$ by Lemma \ref{l:orders}. Together with \eqref{e:dif1}, this yields $u^{3\varphi(i)/2-l}<4d/2^{\varphi(i)/2}$. The values of $3\varphi(i)/2-l$ and $4d/2^{\varphi(i)/2}$ given in the last two columns of the table show that this is not the case if $t(L)>6$.
Suppose that $t(L)=6$. Then we get  $u<7/2$ and $q^6/2d<2u^8$, which yields $q\leq 7$. But then $(7,q\pm 1)=1$ and we can replace $d$ with $1$. Then $u=3$ and $q\leq 5$, contradicting \eqref{e:dif1}).

Now let $S\in\{O_{18}^\pm(u), O_{20}^+(u)\}$. Then $t(L)=6$ and $L$ is one of the groups  $S_{14}(q)$, $O_{15}(q)$, $O_{16}^+(q)$ and $O_{14}^\pm(q)$. Arguing as above, we derive that $k_7(\epsilon q)\in \omega(S)$ and so $k_7(\epsilon q)<2u^{10}$. Also applying Lemmas~\ref{l:J(p,L)} and \ref{l:disjoint}, we conclude that for some suitable $j$, the vertices $r_7(\epsilon q)$ and $r_j(q)$ have disjoint neighborhoods in $GK(S)$.

If $u\neq 2,3,5$, then  by Lemma \ref{l:unique}, it follows that $r_7(\epsilon q)\in R_l(u)$, where $l\in\{9,18,16\}$ and so $k_7(\epsilon q)$ divides $k_9(u)$, $k_{18}(u)$ or $k_{16}(u)$. So $2q^6/21<2u^8$ and together with \eqref{e:dif1}, this yields $8u^9<q^6<21u^8$, a contradiction.

Let $u=2,3$ or $5$. Since $q\geq 5$, Lemma~\ref{l:hering} implies that
$5q^6/(6(7,q-\epsilon))<k_7(\epsilon q)$ and so $q^6<12(7,q-\epsilon)u^{10}/5$. For $u=5$, we have $q\leq 23$ if $(7,q-\epsilon)=7$ and $q\leq 16$ otherwise, so $q\leq 13$, contradicting
\eqref{e:dif1}. Similarly, $q=3$ if $u=2$ and $q\leq 7$ if $u=3$, a contradiction.
\end{proof}

\begin{lemma}\label{l:t(S)=t(L)+1lue}
 If $S=L_m^\tau(u)$, where $m$ is even, then $t(S)\neq t(L)+1$.
\end{lemma}

\begin{proof}
Assume the opposite. We begin with estimating $k_j(q)$ for $j\in J'(p,L)$.
If $L\neq O_{14}^\varepsilon(q)$, then  we choose $\epsilon i\in J'(p,L)$ and $d$ so that $\varphi(i)\leq\varphi(j)$ and $k_{j}(q)\geq q^{\varphi(j)}/2d$ for all $j\in J'(p,L)$. Then $i$ and $d$ are as in Table \ref{tab:t(L)+1lue}.
\begin{table}[h]
\caption{Indices for Lemma \ref{l:t(S)=t(L)+1lue}}\label{tab:t(L)+1lue}
$\begin{array}{|c|c|c|c|c|c|c|c|c|c|c|c|}
\hline
t(L)&m&L&i&d&a&\frac{4d}{2^{\varphi(i)/2}}&L&i&d&a&\frac{4d}{2^{\varphi(i)/2}}\rule{0pt}{12pt} \\[2pt]
\hline
6&14&S_{14},  O_{16}^+&7&7&3&7/2&O^\pm_{14}&7,12&7,1&3,0&7/2,1\\
7&16& O_{20}^+&9&3&1&3/2& O_{18}^\pm&9&3&1&3/2\\
8&18&S_{18} &9&3&3&3/2&S_{20}, O_{20}^-&20&5&6&5/4\\
9&20&S_{22},  O_{24}^+&11&11&7&11/8&O_{22}^\pm&20&11&4&11/4\\
10&22&O_{28}^+&13&13&8&13/16&S_{24}, O_{24}^-, O_{26}^\pm&24&13&2&13/4\\
11&24&S_{26}, &13&13&10&13/16&S_{28}, O_{28}^-&28&1&10&1/16\\
12&26&S_{30}, O_{32}^+&15&1&0&1/4&O_{30}^\pm&15&1&0&1/4\\
13&28&O_{36}^+&17&17&12&17/64&O_{34}^\pm&17&17&12&17/64\\
\hline
\end{array}$
\end{table}

 If $L=O_{14}^\varepsilon(q)$, then $J'(p,L)=\{\varepsilon 7,12\}$ and $k_{j}(q)\geq q^{\varphi(j)}/2d$, where $d=7$ if $j=\varepsilon 7$ and $d=1$ otherwise, and this fact is indicated in the corresponding row of the table.

Suppose first that there is $j\in J'(p,L)$ such that $k_j(q)$ divides $k_{m}(\tau u)$. Then, by the above and Lemma \ref{l:hering}, we have $q^{\varphi(i)}/2d<2u^{\varphi(m)}$. Together with \eqref{e:dif1}, this yields $$2^{\varphi(i)/2}u^{3\varphi(i)/2}<q^{\varphi(i)}<4du^{\varphi(m)},$$ whence $u^a<4d/2^{\varphi(i)/2}$, where $a={3\varphi(i)/2-\varphi(m)}$. The table shows that this is impossible.

Thus we may assume that there is no $j\in J'(p,L)$ such that $k_j(u)$ divides $k_m(\tau u)$. By Lemma \ref{l:lm_even}, it follows that $m=16$, $22$, $26$, or $28$, $u-\tau$ divides $m$, $k_j(q)=k_{m-1}(\tau u)$ for $j\in J'(2,L)$ and $m_j(q)=(u^{m-1}-\tau)/(u-\tau)$. In particular, $u^b<4d/2^{\varphi(i)/2}$ for  $b=3\varphi(j)/2-\varphi(m-1)$. If $m=28$ or $16$, then $b=6$ or $1$, and we can argue as above.

Let $m=22$.  Then $\tau u\neq -2$ since otherwise $m_j(q)$ is divisible by $3$ which is not true by Lemma~\ref{l:lm_even}. So
$r_3(\tau u)$ exists and is adjacent to $r_j(q)$ in $GK(S)$. If $L=O_{28}^+(q)$ or $O_{26}^\pm(q)$ and $l\in J'(p,L)\setminus \{j\}$, then $r_l(q)\in\pi(S)$ and is not adjacent to $r_3(\tau u)$ in $GK(S)$.
By assumption, it is true that $r_l(q)\not\in R_{22}(\tau{u})$,
so $r_l(q)\in R_{20}(\tau{u})$. Therefore, we get that $k_l(q)$ divides $k_{20}(\tau u)$. If $L=S_{24}(q)$, $O_{25}(q)$, or $O_{24}^-(q)$, then
$j=24$. By Lemma~\ref{l:disjoint}(ii), primes $r_{11}(q)$ and $r_{22}(q)$ have disjoint neighborhoods with $r_{24}(q)$ in $GK(L)$, so they are not adjacent to $r_3(\tau u)$ in $GK(S)$. Therefore, there is $\epsilon \in\{+,-\}$ such that $k_{11}(\epsilon q)$ divides $k_{20}(\tau u)$.
In either case, one of the numbers $k_{11}(\pm q)$, $k_{13}(\pm q)$, and $k_{24}(q)$ divides $k_{20}(\tau u)$. Then $q^8<3u^8$, contrary to \eqref{e:dif1}.

Thus we are left with the case $m=26$,  $k_j(q)=k_{25}(\tau u)=(u^{25}-\tau)/(u^5-\tau)$ and  $m_j(q)/k_j(q)=(u^5-\tau)/(u-\tau)$.
Since $L$ is one of $S_{30}(q)$, $O_{31}(q)$,  $O_{32}^+(q)$, and $O_{30}^\pm(q)$, Lemma~\ref{l:disjoint}(iii) implies that $m_j(q)$ is equal to $(q^{15}-\epsilon)/2$, $(q^{15}-\epsilon)/4$ or $(q^{14}+1)(q+\epsilon)/(4,q-\epsilon)$ for some $\epsilon\in\{\pm1\}$.

In the first and second cases, let $r\in R_5(\epsilon q)$. Then $r$ is not adjacent to $r_l(q)$ for $l\in\{28,\pm 13, 24, -\epsilon 11\}$ and so $t(r,L)\geq 6$. By Lemma \ref{l:primes}(ii), this yields $t(r,S)\geq 6$. On the other hand, $r\in R_5(\tau u)$ and $t(r,S)\leq 5$ by \cite[Lemma 3.9]{26Sta.t}, a contradiction.

In the third case, $j=28$ and $m_j(q)=(q^{14}+1)(q+\epsilon)/4$ since $j\in J'(2,L)$ by assumption. Now $k_j(q)=(q^{14}+1)/(q^2+1)$, and so $$\frac{(q^2+1)(q+\epsilon)}{4}=\frac{u^5-\tau}{u-\tau}.$$
Recall $u-\tau$ divides $m$, and so $\tau u\in\{+27,-25,+3, +2\}$. Calculating $k=(u^5-\tau)/(u-\tau)$ for all such $\tau u$ and taking prime divisors of $k$ congruent congruent 1 modulo $4$, we conclude that $u=25$ and $(q^2+1)/2$ divides  $41\cdot 9161$ or $u=27$ and  $(q^2+1)/2=4561$. This yields $u=25$ and $q=9$, a contradiction with (\ref{e:dif1}).
\end{proof}

\begin{lemma}\label{l:t(S)=t(L)+1luo}
 If $S=L_m^\tau(u)$, where $m$ is odd, then $t(S)\neq t(L)+1$.
\end{lemma}

\begin{proof} Assume the opposite.  According to Table~\ref{tab:tL}, we see that $J(S)=E(S)$.

Suppose that $m=15$. Then $L$ is one of $O_{18}^\varepsilon(q)$,  $O_{20}^+(q)$.
Since $J(2,L)\subseteq \{9,18,16\}$ and $J(2,S)\subseteq \{14,15\}$ or $\{7,30\}$, it follows that one of $k_{9}(q)$, $k_{18}(q)$ and $k_{16}(q)$ divides $k_{14}(\tau u)$ or $k_{15}(\tau u)$. Observing that $q\geq 5$ and applying Lemma \ref{l:hering}, we conclude that $4q^6/15<2u^8$. Together with \eqref{e:dif1}, this yields $8u^9<15u^8/2$, a contradiction.

Let $m=21$ and $L$ is one of $S_{24}(q)$, $O_{25}(q)$, $O_{24}^-(q)$. Then $24\in J(2,L)$ and so $k_{24}(q)$ divides $k_{20}(\tau u)$ or $k_{21}(\tau u)$. Hence   $4q^{8}/5<2u^{12}$, contradicting \eqref{e:dif1}.

In other cases, we use the following argument. Suppose that $I\subseteq E(L)$,  $l=|I|$ and set $k(I)=\min_{i\in I} k_i(q)$.
We also denote by $\varphi(I)$ the multiset $\{\varphi(i)\mid i\in I\}$. Enumerate the elements of $J(S)$ so that $\varphi(j_1)\geq \varphi(j_2)\geq \dots$  Suppose that $i\in I$ and every $r_i(q)$ is large with respect to $S$. By Lemma \ref{l:trans}, there is $j\in J(S)$ such that $k_i(q)$ divides $k_j(u)$ and so $k_i(q)\leq 2u^{\varphi(j)}$.  It follows that $k(I)\leq 2u^{\varphi(j_l)}$. If there is $i\in I$ such that $r_i(q)$ is not large with respect to $S$, then Lemma \ref{l:lm_odd} implies that $k(I)\leq 2u^{\varphi(j'_{l-1})}$, where $j_1'$, $j_2'$ are elements of $J(S)\setminus\{m-1\}$ or $J(S)\setminus\{\nu(m-1)\}$ enumerated so that $\varphi(j_1')\geq \varphi(j_2')\geq \dots$

Suppose that $t(L)=6$. Then $m=13$ and $L$ is one of $S_{14}(q)$, $O_{15}(q)$, $ O_{16}^+(q)$ or $L=O_{14}^\varepsilon(q)$. Set $I=\{14,7,12,10,5\}$ in the first case and $I=\{\varepsilon 7,12,10,5,8\}$ in the second one. Then $I\subseteq E(L)$ and $k(I)\geq \frac{4}{5}\cdot\frac{q^4}{5}$.  Since $\varphi(\{7,8,\dots,13\})=\{12,10,6,6,4,4,4\}$,
$\varphi(\{13,11,10,9,8\})=\{12,10,6,4,4\}$, it follows that $4q^4/25<2u^4$. Together with (\ref{e:dif1}), this yields $4u^6<25u^4/2$,  a contradiction.

Suppose that $S=O_{20}^-(q)$. Then $t(L)=8$ and $m=17$. Set $I=\{20,9,18,16,7,14\}$. Then $I\subseteq E(L)$ and $k(I)\geq q^6/14$. Also
$\varphi(\{9,10,\dots,17\})=\{16,12,10,8,8,6,6,4,4\}$ and
$\varphi(\{10,\dots,17\}\setminus\{16\})=\{16,12,10,8,6,4,4\}$, and we see that $4q^6/35<2u^6$, contrary to (\ref{e:dif1}).

Thus $t(L)\geq 8$ and $S\neq O_{20}^-(q)$. We can proceed as above,
but a shorter way is to prove first that $r_i(q)$ is always large with respect to $S$ whenever $i\in E(L)$ and $r_i(q)r_8(q)\not \in \omega(G)$. Using Table \ref{tab:tL}, it is not hard to verify that $8\not\in J(L)$ and there are distinct $i_1\in E(L)$, $i_2\in J(L)$ such that $s_1=r_{i_1}(q)$, $s_2=r_{i_2}(q)$ are adjacent to $r_8(q)$ in $GK(L)$.

By Lemmas \ref{l:primes}(i) and \ref{l:R8}, we have $R_8(q)\cap (\pi(K)\cup\pi(\overline G/S))=\varnothing$. Suppose that $i\in J(L)$,  $r_i(q)r_8(q)\not\in \omega(L)$ and some $r\in R_i(q)$ is not large with respect to $S$. By Lemmas \ref{l:primes}(iv) and \ref{l:lm_odd}, it follows $r_8(q), s_1, s_2$ are large with respect to $S$.
Since $s_1r_8(q)\in \omega(S)$ and $J(S)=E(S)$, we see  that $s_1,r_8(q)\in R_j(\tau u)$ for some $j$. The same is true for $s_2$ and $r_8(q)$, and we get a~contradiction since $s_1s_2\not\in\omega(L)$.

In Table \ref{tab:t(L)+1luo}, for every $L$ under consideration, we give $I$ and $d$ such that
$I\subseteq \{i\in E(L)\mid r_8(q)r_i(q)\not\in \omega(L)\}$ and $k_i(q)\geq \Phi_i(q)/d$ for every $i\in I$. Writing $\pm i$ for odd $i$, we mean that both $i$ and $2i$ lie in $I$. Choose $i\in I$ with minimal $\varphi(i)$ and set $j=j_l$. Then \begin{equation}\label{e:k(I)} 4q^{\varphi(i)}/5d<k(I)<2u^{\varphi(j)},\end{equation} and using  \eqref{e:dif1}, we conclude that $u^a<5d/(2\cdot 2^{\varphi(i)/2})$, where $a=3\varphi(i)/2-\varphi(j)$.

\begin{table}[h]
\caption{Indices for Lemma \ref{l:t(S)=t(L)+1luo}}\label{tab:t(L)+1luo}
$\begin{array}{|c|c|c|c|c|c|c|c|c|c|c|c|}
\hline
t(L)&m&\varphi(j_1),\varphi(j_2),\dots&L&I&d&a&\left\lfloor{\frac{5d}{2\cdot 2^{\varphi(i)/2}}}\right\rfloor\rule{0pt}{12pt} \\[2pt]
\hline
8&17&16,12,10,8,8&S_{18}, S_{20} &\pm 9,16,\pm 7&7&1&2\\
9&19&18,16,12,10,8&S_{22},  O_{22}^\varepsilon&\varepsilon 11, 20, \pm9, 16&11&1&3\\
  &  &           &O_{24}^+&\pm 11, 20, \pm 9&11&1&3\\
10&21&18,16,12& O_{26}^\varepsilon, O_{28}^+&\varepsilon13, \pm 11&13&3&1\\
11&23&22,18,16,12&S_{26}, S_{28}, O_{28}^- &\pm 13, \pm 11&13&3&1\\
12&25&22,20,18&S_{30}, O_{30}^\pm, O_{32}^+&28, \pm 13&13&0&0\\
13&27&22,20&O_{34}^\varepsilon, O_{36}^+&\varepsilon17,32&17&4&0\\
\hline
\end{array}$
\end{table}

The table shows that either $m=17$ and $u=2$, or $m=19$ and $u\leq 3$. It follows from \eqref{e:k(I)} that $q<5$ if $u=2$, and $q\leq 7$ if $u=3$, which contradicts \eqref{e:dif1}.
\end{proof}

\section{Proof of Theorem \ref{t:main}: Case $t(S)=t(L)$}

We keep notation and agreements of the previous two sections. In particular, we omit the groups $O_{2n+1}(q)$ in the tables and write $-i$ instead of $2i$ for odd indices $i$.

If $t(S)=t(L)$, we does not have the inequality \eqref{e:dif1}. We will replace it with the following argument. Suppose that for some $i\in J'(p,L)$, we have proved that there is $j\in J'(v,S)$ such that $R_i(q)\subseteq R_j(u)$. Then $k_i(q)$ divides $k_j(u)$. Since $k_j(u)$ divides $m_j(u)$, Lemma \ref{l:disjoint}(ii) implies that $m_j(u)$ divides $m_i(q)$, and this is a replacement for \eqref{e:dif1}.

\begin{lemma}\label{l:t(S)=t(L)lue} If $S=L_m^\tau(u)$, where $m$ is even, then $t(S)\neq t(L)$.\end{lemma}

\begin{proof}
Assume the opposite. Suppose that there is $i\in J'(p,L)$ such that $R_i(q)\subseteq R_m(\tau u)$. Then $k_i(q)$ divides $k_m(\tau u)$ and
$(u^m-\tau)/((u-\tau)(m,u-\tau))$ divides $m_i(q)$.
In Table \ref{tab:t(L)lue}, we give possible values of $i$  and also $d$ such that $k_i(q)\geq \Phi_i(q)/d$ (for brevity, if $t(p,L)=\{n,2n\}$, then we write only $n$, and  if $t(p,L)=\{\varepsilon n, 2n-2\}$, we write $n$ and $2n-2$).

\begin{table}[h]
\caption{Indices for Lemma \ref{l:t(S)=t(L)lue}}\label{tab:t(L)lue}
$\begin{array}{|c|c|c|c|c|c|c|c|c|c|c|c|}
\hline
t(L)&m&L&i&d&a&b&L&i&d&\min a&\max b\\
\hline
6&12&S_{14},  O_{16}^+&7&7&19/3&35&O^\pm_{14}&7,12&7,1&4&35\\
7&14& O_{20}^+&9&3&4&27 &O_{18}^\pm&9,16&3,2&4&27\\
8&16&S_{18} &9&3&3&27&S_{20}, O_{20}^-&20&5&5&30\\
9&18&S_{22},  O_{24}^+&11&11&52/5&47&O_{22}^\pm&11,20&11,5&35/4&47\\
10&20&O_{28}^+&13&13&31/3&53&S_{24}, O_{24}^-, O_{26}^\pm&13,24&13,1&6&53\\
11&22&S_{26} &13&13&61/6&53&S_{28}, O_{28}^-&28&1&28/3&4\\
12&24&S_{30}, O_{32}^+&15&1&8&8&O_{30}^\pm&15,28&1,1&8&8\\
13&26&O_{36}^+&17&17&49/4&66& O_{34}^\pm&17,32&17,2&49/4&66\\
\hline
\end{array}$
\end{table}

It follows that \begin{equation}\label{e:le} 2q^{\varphi(i)}/3d<2u^{\varphi(m)}.\end{equation}
According to Table~\ref{tab:mi}, we see that $m_i(q)<q^n$.
On the other hand, $$\frac{u^m-\tau}{(u-\tau)(m,u-\tau)}\geq\frac{u^m+1}{(u+1)(m,u-\tau)}\geq\frac{2u^{m-1}}{3(m,u-\tau)}.$$
Therefore,  $u^{m-1}\leq 3(m,u-\tau)q^n/2$ and hence
$$ u^{(m-1)-n\varphi(m)/\varphi(i)}<3(m,u-\tau)(3d)^{n/\varphi(i)}/2. $$

Set  $a=(m-1)-n\varphi(m)/\varphi(i)$ and $b=\left\lceil{(3d)^{n/\varphi(i)}}\right\rceil$.

Suppose that $a\geq 8$. It follows from the table that $u^8<3\cdot 13\cdot 66$, which yields $u=2$. If $m\neq 18$, then
$(u-\tau,m)=1$ and hence $u^8<3\cdot 33$, while if $m=18$, then $u^{8}<(9/2)\cdot 47$, a contradiction. Suppose that $6\leq a<8$. Then $u^6<3\cdot 11\cdot 53$, whence $u\leq 3$ and $(m,u-\tau) \leq 4$. So $u^6<6\cdot 53$ and therefore $u=2$. Similarly, $u=2$ if $a=5$ and $u\leq 3$ if $a=4$.
Now \eqref{e:le} implies that $q<3$ if $u=2$ and $q<5$ if $u=3$, a contradiction with $q\neq u$.

If $a=3$, then $m=16$,  $i\in\{9,18\}$  and $b=27$. The inequality $u^3<3(m,u-\tau)\cdot27$ implies that $u\leq7$. Recall that $k_i(q)$ divides $k_m(\tau u)$, so  $k_{16}(u)=(u^8+1)/(2,u-1)$ must have a prime divisor congruent to $1$ modulo $9$. If $u=2$ or $4$, then $k_{16}(u)$ is a Fermat prime. If $u=3,5,7$, then  $k_{16}(u)\in\{17\cdot193,17\cdot11489, 17\cdot169553\}$, a contradiction.

Thus we may assume that there is no $i\in J'(p,L)$ such that $R_i(q)\subseteq R_m(\tau u)$. By Lemma \ref{l:lm_even}, it follows that $m\in\{16,22,26\}$, $u-\tau$ divides $m$, $k_j(q)=k_{m-1}(\tau u)$ and $m_j(q)=(u^{m-1}-\tau)/(u-\tau)$ for $j\in J'(2,L)$.

If $m=22$, then $L$ is one of $S_{26}(q)$, $O_{27}(q)$, $S_{28}(q)$, $O_{29}(q)$, $O_{28}^-(q)$ and so $j\in\{13,26,28\}$ by Lemma~\ref{l:J(p,L)}.
Also $\tau{u}\in\{2,3,23\}$ because $u-\tau$ divides 22. Since $k_{21}(2)=337$, $k_{21}(3)=368089$, and $k_{21}(23)=43\cdot170689\cdot408030421$, and
all primes divisors of $k_{21}(\tau u)=k_j(q)$ are congruent to 1 modulo $j$, it follows that $j=28$ and $k_{21}(\tau u)\in\{337, 368089\}$. On the other hand, $k_{28}(3)=478297$ and $k_{28}(q)=\Phi_{28}(q)>4/5\cdot 5^{12}>368089$ for $q\geq5$, a contradiction.

If $m=26$, then $L=O_{34}^\pm(q)$ or $O_{36}^+(q)$ and  $j\in\{17,34,32\}$.
Since $u-\tau$ divides 26, we find that $\tau{u}\in\{2,3,-25,27\}$.
Now $601$ divides $k_{25}(2)$ and $k_{25}(27)$, $8951$ divides $k_{25}(3)$, and $424256201$ divides $k_{25}(-25)$, contradicting the fact that every prime divisor of $k_{25}(\tau u)$ must be congruent to 1 modulo $j$.

If $m=16$, then $L$ is one of the groups $S_{18}(q)$, $O_{19}(q)$, $S_{20}(q)$, $O_{21}(q)$, $O_{20}^-(q)$ and $j\in\{9,18,20\}$.
According to Table~\ref{tab:mi}, we see that $m_j(q)=(q^9\pm1)/2$ or $(q^{10}+1)/2$. Therefore, $2m_{15}(\tau{u})\pm1\in\{q^9,q^{10}\}$.
Since $u-\tau$ divides 16, we get that $\tau{u}\in\{2,3,-3,5,-7,9,17\}$.
Now it is easy to verify that  $2m_{15}(\tau{u})\pm1\neq q^9,q^{10}$.
\end{proof}

\begin{lemma}
If $S$ is one of the groups $O_{2m+1}(u), S_{2m}(u), O_{2m+2}^+(u)$, $O_{2m}^\tau (u)$, where $m$ is odd,
then $t(S)\neq t(L)$.
\end{lemma}

\begin{proof}
Assume the opposite. Recall that $J(v,S)$ are equal to $\{m,2m\}$ or $\{\tau m, 2m-2\}$ by Lemma~\ref{l:J(p,L)}.

Let $L$ be one of the groups $S_{2n}(q)$, $O_{2n+1}(q)$, or $O_{2n}^-$ with $n$ even. Then $n=10$, $12$, or $14$. By Lemma \ref{l:s2n_even}(i), either $u=2,3$ or $S=O_{2m+2}^+(5)$. Since $2n\in J(2,L)$, it follows that there is $j\in \{m,2m, 2m-2\}$ such that $k_{2n}(q)$ divides $k_j(u)$ and, in particular, $k_j(u)$ has a prime divisor congruent to $1$ modulo $2n$.

Suppose that $n=10$. Using Table~\ref{tab:6tL14}, we see that $S$ is one of the groups $O_{19}(u)$, $S_{18}(u)$, where $u\in\{2,3\}$. Then $j=9$ or 18,
so $k_{20}(q)$ divides $k_9(2)=73$, $k_{18}(2)=19$,  $k_9(3)=757$, or
$k_{18}(3)=19\cdot 37$, a contradiction.

Suppose that $n=12$. Then $S$ is one of the groups $O_{26}^\pm(u)$, where $2\leq u\leq3$, or $O_{28}^+(u)$, where $u\in\{2,3,5\}$, and $j\in\{13,26,24\}$. If $j=13,26$, then $k_j(u)$ has a prime divisor that is congruent to 1 modulo 24 only if $j=13$ and $u=3$.
Since $k_{13}(3)=797161$ is prime, we have
$q^8-q^4+1=k_{24}(q)=797161$, which is impossible. If $j=24$, then $k_{24}(q)$ divides $k_{24}(u)$ and $k_{24}(u)$ is a prime from the set $\{241,6481,390001\}$. If $q>u$,
then $k_{24}(q)>k_{24}(u)$, while if $q<u$, then
$k_{24}(q)$ does not divide $k_{24}(u)$, a contradiction.

Suppose that $n=14$. It follows that $S$ is one of the groups $O_{27}(u)$, $S_{26}(u)$, where $u\in\{2,3\}$. Then $j=13$ or $26$,
so $k_{28}(q)$ divides $k_{13}(2)=8191$, $k_{26}(2)=2731$,  $k_{13}(3)=797161$, or $k_{18}(3)=398581$.
On the other hand, $k_{28}(3)=478297$ and $k_{28}(q)>(4/5)5^{12}=195312500$ for $q\geq5$, a contradiction.

Let $L$ be one the groups $S_{2n}(q)$, $O_{2n+1}(q)$, $O_{2n+2}^+(q)$ or $L=O_{2n}^\varepsilon(q)$, where $n$ is odd. Using Table \ref{tab:6tL14}, we get that $7\leq n\leq 17$ and $m=n$. Recall that $J(p,L)$ and $J(v,S)$ are equal to $
\{n,2n\}$ or $\{\varepsilon n, 2n-2\}$. By Lemma \ref{l:reduction}(ii), if $i\in J(2,L)$,  then there is $j\in J(2,S)$ such that $k_i(q)$ divides $k_j(u)$ and $m_j(u)$ divides $m_i(q)$.

If $i\in \{n,2n\}$ and $j\in \{n,2n\}$, then $k_n(\epsilon_1q)$ divides $k_n(\epsilon_2 u)$ for some $\epsilon_1, \epsilon_2\in\{+,-\}$ and $(u^n-\epsilon_2)/(4,u-\epsilon_2)$ divides $(q^n-\epsilon_1)/2$, contradicting Lemma \ref{l:tight}(ii). In particular, at least one of the sets  $J(p,L)$ and $J(v,S)$ is not $\{n,2n\}$.

If $u>3$ and $S\neq O_{2n}^+(5)$, then by Lemmas~\ref{l:disjoint}(i) and \ref{l:unique}, for every $i\in J(p,L)$, there is $j\in J(v,s)$ such that $k_i(q)$ divides $k_j(u)$. So either we are in the situation of the preceding paragraph, or $L=O_{2n}^\varepsilon(q)$, $S=O_{2n}^\tau(u)$, and $k_{2n-2}(q)$ divides $k_n(\tau u)$, $k_n(\varepsilon q)$ divides $k_{2n-2}(\tau u)$ and  $(u^n-\tau)/(4,u-\tau)$ divides $(q^{n-1}+1)(q+\varepsilon)/(4,q-\varepsilon)$. This contradicts Lemma \ref{l:tight}(iii) if $n\neq 17$.

Let $n=17$. Then $k_{32}(q)$ divides $k_{17}(\tau u)$  and $(u^{17}-\tau)/(4,u-\tau)$ divides $(q^{16}+1)(q+\varepsilon)/(4,q-\varepsilon)$. The latter implies that $u^{17}<3q^{17}$. If $(17,u-\tau)=17$ or $k_{32}(q)$ is a proper divisor of $k_{17}(\tau u)$, then $q^{16}/2<2u^{16}/17$, which contradicts the previous inequality. Hence $(q^{16}+1)/2=(u^{17}+\tau)/(u+\tau)$ and subtracting 1 from the both sides yields $$\frac{q^{16}-1}{2}=\frac{u(u^{16}-1)}{u+\tau}.$$ Note that $r_{16}(q)$ is large with respect to $L$, and so with respect to $S$. On the other hand, the only prime divisors of the right side that are large with respect to $S$ are $r_{16}(u)$. So  $R_{16}(q)\subseteq R_{16}(u)$ and $k_{16}(q)$  divides $k_{16}(u)$. It is easily seen that $k_{16}(q)\neq k_{16}(u)$, and therefore $q^8/2<2u^8/17$, which contradicts to the inequality $u^{17}<3q^{17}$.

Thus we are left with the case when either $u=2, 3$, or $S=O_{2n}^+(5)$.
Suppose first that $\epsilon n$ belongs to $J(2,L)$ for some $\epsilon\in\{+,-\}$. By the above discussion, we may assume that $k_n(\epsilon q)$ divides  $k_{2n-2}(u)$.

If $n\neq9,15$, then $\varphi(n)\geq\varphi(2n-2)$, so $(q-1)q^{\varphi(n)-1}<2u^{\varphi(n)}$. If $q>u$, we get a contradiction since $(q-1)q^2>2u^3$ for $u\in\{2,3,5\}$. If $q<u$, then $q=3$ and $u=5$, so $m_{2n-2}(5)=(5^{n-1}+1)(5+1)/4$ is divisible by 3 and hence does not divide $m_n(\epsilon3)$.

If $n=9$, then $k_{2n-2}(u)\in\{257, 17\cdot 193, 17\cdot 11489\}$, so $k_{2n-2}(u)$ is not divisible by $k_n(\epsilon q)$.
If $n=15$, then $k_{2n-2}(u)\in\{29\cdot113, 29\cdot 16493,
234750601\}$. We see that only $234750601$ is congruent to 1 modulo 15
and if $k_{15}(\epsilon q)=234750601$,
then $q^4-1$ divides $\Phi_{15}(\epsilon q)-1=k_n(\epsilon q)-1=234750600$. This is a contradiction since $16$ divides $q^4-1$.

Thus $J(2,L)=\{2n-2\}$, and so $L=O_{2n}^\varepsilon(q)$, where $q\equiv\varepsilon\pmod{8}$.
In particular, $q\geq7>u$. Let $n'\in\{n,2n-2,2n\}$ such that $k_{2n-2}(q)$ divides $k_{n'}(u)$. If $\varphi(2n-2)\geq\varphi(n')$, then
$$k_{2n-2}(q)>(2/3)7^{\varphi(n')}>2\cdot5^{\varphi(n')}>k_{n'}(u).$$ Therefore,
we find that $n=7,11,13$ and $n'\in\{n,2n\}$. If $u\leq 3$,
then $$k_{2n-2}(q)>(6/7)7^{\varphi(2n-2)}>2\cdot3^{\varphi(n)}>k_{n'}(u).$$
So $u=5$. If $n=7$, then $k_{n'}(u)\in\{19531, 29\cdot449\}$
and all involved primes are not congruent to 1 modulo 12.
By a similar reason $n\neq 13$, for otherwise $k_{n'}(u)\in\{ 5227\cdot 38923, 305175781\}$. Finally, if $n=11$, then $k_{n'}(u)\in\{12207031, 23\cdot 67\cdot 5281\}$, so $k_{20}(q)=5281$ which is impossible since $k_{20}(q)>(6/35)7^8$.
\end{proof}

\begin{lemma}\label{l:t(S)=t(L)se}
If $S=S_{2m}(u), O_{2m+1}(u),  O_{2m}^-(u)$, where $m$ is even, then $t(S)\neq t(L)$.
\end{lemma}

\begin{proof}
Assume the opposite. Then $t(S)\geq 7$ and $8\leq m\leq 16$. If $L\in\{S_{2n}(q), O_{2n+1}(q), O_{2n}^-\}$, where $n$ is even,
then $n=10,12,14$, $m=n$ and by Lemma \ref{l:s2n_even}, it follows that $k_{2n}(q)$ divides $k_{2n}(u)$, while $(u^n+1)/(2,u-1)$ divides $(q^n+1)/2$. This contradicts Lemma \ref{l:tight}(i). So $L$ is one of the groups $S_{2n}(q)$, $O_{2n+1}(q)$, $O_{2n+2}^+(q)$ or $L=O_{2n}^\varepsilon(q)$, where $n$ is odd. By Table \ref{tab:6tL14}, we see that $m=n-1$ and $L\neq S_{2n}(q), O_{2n+1}(q)$ if $m\equiv 0\pmod 4$, and $m=n+1$ and $L=S_{2n}(q), O_{2n+1}(q)$ otherwise.

By Lemma~\ref{l:J(p,L)}, we have $|J(p, L)|=2$. It follows from
Lemmas \ref{l:disjoint}(i) and \ref{l:unique}(v) that there is $i\in t(p,L)$ such that $k_i(q)$ divides $k_{2m}(u)$. Then $(u^m+1)/(2,u-1)$ divides $m_i(q)$, which is equal to $(q^n\pm 1)/2$  or $(q^{n-1}+1)(q+\varepsilon)/(4,q-\varepsilon)$ according to Table~\ref{tab:mi}. Note that $(u^m+1)/(2,u-1)$ is odd and hence divides $(q^{n-1}+1)(q+\varepsilon)/4$ even if $(4,q-\varepsilon)=2$. Also $u^m\neq (q^n-1)/2$ if $u$ is even, so \begin{equation} \label{e:ts=tl_se} u^m\leq \frac{(2,u-1)q^n}{2}.\end{equation} In particular, $u=2$ if $q=3$.

Let $m=14$. Then $n=13$ and $k_{13}(\epsilon q)$ divides $k_{28}(u)$ for some $\epsilon \in\{+,-\}$. Since $k_{13}(\pm 3)>k_{28}(2)$, we have $u\neq 2$ and $q\geq 5$. So $4q^{12}/5d<k_{13}(\epsilon q)\leq k_{28}(u)<u^{12}$, where $d=(13,q^2-1)$.  It follows that $u^{14}<(5d/4)^{13/12}u^{13}$, which yields $u\leq 20$. Then $q<25$, so $d=1$ and $u<2$.

For $m\neq 14$, we use the argument from the proof of Lemma \ref{l:t(S)=t(L)+1luo}. Let $I$ and $d$ be as in Table~\ref{tab:t(L)se} (with $7^*$ standing for $\epsilon 7$ such that $(7,q-\epsilon)=1$) and $k(I)=\min_{i\in I} k_i(q)$.
It is easy to see using Table~\ref{tab:tL} that $I\subseteq E(L)$. By Lemma \ref{l:trans}, for every $i\in I$, either $k_i(q)$ divides $k_j(u)$ for some $j\in E(S)$, or $m=10$ and $k_i(q)$ divides $k_{j_1}(u)k_{j_2}(u)$, where $j_1,j_2\in\{5,10,8\}$. In the last case, we set $j=\{j_1,j_2\}$ and $\varphi(j)=8$ and observe that $k_{j}(u)$ is less then $3u^8/2$ unless $u=2$. But if $u=2$, then $q=3$, $k_j(u)\leq 527$ and $k_i(q)\geq 547$, a contradiction. So we can assume that $u\geq3$ in this case.
Let $l=|I|$ and  $j_1$, $j_2$, \dots be elements of $E(S)$ or $E(S)\cup\{{j_1,j_2\}}$ enumerated so $\varphi(j_1)\geq \varphi(j_2)\geq \dots$.

\begin{table}[h]
\caption{Indices for Lemma \ref{l:t(S)=t(L)se}}\label{tab:t(L)se}
$\begin{array}{|c|c|c|c|c|c|c|c|c|c|c|c|}
\hline
m&\varphi(j_1),\dots,\varphi(j_l)&L&I&d&L&I&d&a&b&u&q\\
\hline
8&8,6,6,4,&O_{20}^+&\pm 9, 16,7^*&3&O_{18}^\varepsilon&\varepsilon 9,16,\pm 7&7&2&24&u\leq 5&q=3\\
10&8,8,8,6&S_{18}&\pm 9,16,7^*&3&-&-&-&1&7&u\leq 13&q\leq 17\\
12&10,10,8&O_{28}^+&\pm 13, 11&11&O_{26}^\varepsilon&\varepsilon13,\pm 11&11&8/5&26&u\leq 11&q\leq 9\\%
16&16,12&O_{36}^+&\pm 17&17&O_{34}^\varepsilon&\varepsilon 17,32&17&13/4&20&u=2&\\
\hline
\end{array}
$
\end{table}

It is not hard to verify that $k(I)>2u^{\varphi(j_l)}$ if $u=2$, so we may assume that $u\geq 3$ in all cases, so $q\geq 5$ and $k_{j_l}(u)<(3/2)u^\varphi(j_l)$. Taking $i\in I$ with minimal $\varphi(i)$ and setting $j=j_l$, we conclude that
$4q^{\varphi(i)}/5d<3u^{\varphi(j)}/2$. Together with \eqref{e:ts=tl_se}, this yields
$$u^m\leq (2,u-1)q^n/2< ((2,u-1)/2)\cdot (15d/8)^{n/\varphi(i)}u^{n\varphi(j)/\varphi(i)},$$
whence $u^a<(2,u-1)b$, where $a=m-n\varphi(j)/\varphi(i)$ and $b=\left\lceil{(15d/8)^{n/\varphi(i)}/2}\right\rceil$.
Solving this equation, we get that $u$ is as stated in the corresponding column of Table \ref{tab:t(L)se}, in particular, $m\neq 16$. Then $q^{\varphi(i)}<5du^{\varphi(j)}/2$ implies that  $q$ is as stated in Table \ref{tab:t(L)se}, and $m\neq 8$ since we assume that $u>2$. If $m=12$, then
inequality $u\leq 11$ implies that we can take $d=1$, so new estimates
yield $u<2$, a contradiction.

Thus we are left with the case when $m=10$, $u\leq 13$ and $q\leq 17$.
Recall that $k_i(q)$ divides $k_{20}(u)$, where $i\in\{9,18,16\}$ by Lemma~\ref{l:J(p,L)}. This implies that every prime divisor of $k_i(q)$ is congruent to 1 modulo 20. Sorting out all possible $q$, we find that
$i=16$ and $q=13$ or $q=9$.
Now it is easy to see that $k_i(q)$ does not divide $k_{20}(u)$.
\end{proof}

\begin{lemma}\label{l:t(S)=t(L)luo}
If $S=L_m^\tau(u)$, where $m$ is odd, then $t(S)\neq t(L)$.
\end{lemma}

\begin{proof}
Assume the opposite. Let $L$ be one of the groups $S_{2n}(q)$, $O_{2n+1}(q)$, $O_{2n}^-(q)$, where
$n$ is even. Then $n\in\{10,12,14\}$ and $m=15,19,21$, respectively. Since $2n\in J(2,L)$, it follows that $k_{2n}(q)$ divides $k_j(\tau u)$ for $j\in\{m,m-1\}$ and so $m_j(\tau u)$ divides $(q^n+1)/2$. The latter yields $u^{m-1}<(m,u-\tau)q^n$

If $m=15$, then $2q^8/15<2u^8$ and so $u^{14}<(15,u-\tau)(15)^{5/4}u^{10}$. It follows that $u=2$ and then $q<3$. Similarly, if $m=21$, then $2q^{12}/3<2u^{12}$, $u^6<(21,u-\tau)3^{7/6}$ and $u<2$.

Let $m=19$. Then $I=\{24, \pm 11, 20,16\}\subseteq J(L)$. Applying Lemma \ref{l:trans} and noting that $\varphi(J(S))=\{18, 6, 16, 8$, $8$, $6, 12, 4, 10,4\}$,
we derive that there are $i\in I$ and $j\in J(S)$ such that $\varphi(j)\leq 8$ and $k_i(q)$ divides $k_j(u)$. Then  $2q^8/33<2u^8$ and so $u^{18}<(19,u-\tau)33^{3/2}u^{12}$, whence $u=2$ and $q=3$. In this case, $k_i(q)\geq 2q^8/15$, which leads to a contradiction.

Let $L$ be one of the groups $S_{2n}(q)$, $O_{2n+1}(q)$, $O_{2n+2}^+(q)$ or $O_{2n}^\varepsilon(q)$, where $n$ is odd. Let $i\in J(2,L)$. Then  $k_i(q)$ divides $k_j(\tau u)$ for  $j\in\{m,m-1\}$ and so $m_j(\tau u)$ divides $(q^n-\epsilon)/2$ for some $\epsilon$ or $(q^{n-1}+1)(q+\varepsilon)/4$. In either case,  $u^{m-1}<(m,u-\tau)q^n$.

If $m=21$, then $n=13$, $i\in\{13,26\}$ and so $2q^{12}/39<2u^{12}$. It follows that $u^{20}<(21,u-\tau)(39)^{13/12}u^{13}$, whence $u^7<(21,u-\tau)(52)^{13/12}$ and $u=2$. Then $q<3$, a contradiction. Similarly, if $m=25$, then $n=17$,  $i\in\{17,34, 32\}$ and hence $2q^{16}/51<2u^{20}$. We have $u^{24}<(25,u-\tau)(51)^{17/16}u^{85/4}$, which yields $u\leq 4$. Then $q\leq 7$ and $q^{16}/2<2u^{20}$.

For $m\neq 21$, we argue as in the proof of Lemma \ref{l:t(S)=t(L)se}. Let $I$ and $d$ be as in Table \ref{tab:t(L)luo} (as before, $r^*$ stands for $\epsilon r$ such that $(r,q-\epsilon)=1$) and $k(I)=\min_{i\in I} k_i(q)$. Then $k(I)>2q^{\varphi(i)}/3d$, where $i\in I$ has minimal $\varphi(i)$, and  $k(I)<2u^{\varphi(j_l)}$, where $l=|I|$ and  $j_1$, $j_2$, \dots are elements of $E(S)$ enumerated so $\varphi(j_1)\geq \varphi(j_2)\geq \dots$.

\begin{table}[h]
\caption{Indices for Lemma \ref{l:t(S)=t(L)luo}}\label{tab:t(L)luo}
$\begin{array}{|c|c|c|c|c|c|c|c|c|c|c|c|}
\hline
t(L)&m&\varphi(j_1),\varphi(j_2),\dots&L&I&d&a&b\\
\hline
6&11&10,6,6,4&S_{14}, O^+_{16}, O_{14}^\varepsilon&\varepsilon 7,12,5^*,8&1&3&7\\
7&13&12,10,6&O_{20}^+,  O_{18}^\varepsilon&\varepsilon 9,16,7^*&3&3&27\\
8&15&12,10,8,6&S_{18}&\pm 9,16,7^*&3&5&27\\
9&17&16,12,10&S_{22}, O_{24}^+, O_{22}^\varepsilon &\varepsilon 11,20,16&5&9/4&42\\
10&19&18,16,12,10,8& O_{26}^\varepsilon, O_{28}^+&\varepsilon 13,24,\pm 11,16&1&5&6\\
12&23&22,18,16,12&S_{30}, O_{30}^\pm, O_{32}^+&\pm 13, 28, 11^*&1&4&6\\
\hline
\end{array}$
\end{table}

Set $j=j_l$. Then $u^{m-1}<(m,u-\tau)(3d)^{n/\varphi(i)}u^{\varphi(j)n/\varphi(i)}$,
whence $u^a<(m,u-\tau)b$, where $a={m-1-\varphi(j)n/\varphi(i)}$ and $b=\left\lceil{(3d)^{n/\varphi(i)}}\right\rceil$. It follows that either $m\leq 15$ and $u=2$, or $m=17$ and $u=16$ or $u\leq 5$.
If $u\neq 16$, then $q^{\varphi(i)}<3du^{\varphi(j)}$ implies that $q<3$ in the first case, and $q\leq 9$ in the second one.

Therefore we are left with the case when $m=17$, $q\leq 9$ and either $u=16$ or $u\leq 5$. We will use the fact that for every $i\in t(2,L)$, there is $j\in t(2,S)$ such that $k_i(q)$ divides $k_j(u)$. By \cite[Tables 4 and 6]{05VasVd.t}, we have that $\epsilon 11\in J(2,L)$ for some $\epsilon\in\{+,-\}$ unless $L=O_{22}^\varepsilon(q)$ and $q\equiv \varepsilon1\pmod 8$, in which case  $J(2,L)=\{20\}$.  Also $J(2,S)\subseteq \{16,\tau 17\}$. Calculating $k_{11}(\pm q)$ for $q\leq 9$ and $k_{20}(q)$ for $q=7,9$, and leaving only those with all prime divisors congruent to 1 modulo $17$ (or $16$), we derive that $q=9$ and $i=20$, with $k_{20}(9)=42521761$. It is easy to verify that $k_{20}(9)\not\in \pi(S)$, and this contradiction completes the proof.
\end{proof}

We are now in a position to prove  Theorem \ref{t:main}. Let $G$ be a finite group that is isospectral to $L$ but not an almost simple group with socle $L$. Recall that by Lemma \ref{l:reduction}, the group $G$ has a unique nonabelian composition factor $S$ and  $S$ is a classical group in characteristic other than the defining characteristic of $L$. Furthermore, $t(S)\in\{t(L)+1, t(L)\}$ by Lemma \ref{l:rest}.  The case $t(S)=t(L)+1$ is not possible by Lemmas \ref{l:t(S)=t(L)+1so}--\ref{l:t(S)=t(L)+1luo}, while the case $t(S)=t(L)$ is excluded in Lemmas \ref{l:t(S)=t(L)lue}--\ref{l:t(S)=t(L)luo}, and so the theorem follows.

\section*{Acknowledgments}

M. A. Grechkoseeva and A. V. Vasil’ev acknowledge the support of the Russian Science Foundation, Project No. 24-11-00127, https://rscf.ru/en/project/24-11-00127/.

\end{document}